\documentclass[12pt]{amsart}
\usepackage{amsmath,amssymb}
\usepackage{fullpage}

\usepackage[all]{xy}

\title[Real orbits on spherical spaces: split case]
{Real orbits of complex spherical homogenous spaces: the split case}

\author[S.~Cupit-Foutou]{St\'ephanie Cupit-Foutou}
\address{Ruhr-Universit\"at Bochum, Fakult\"at f\"ur Mathematik, D-44780 Bochum, Germany}
\email{stephanie.cupit@rub.de}

\author[D.~A.~Timashev]{Dmitry A. Timashev}
\address{Lomonosov Moscow State University, Faculty of Mechanics and
Mathematics, Department of Higher Algebra, 119991 Moscow, Russia}
\email{timashev@mccme.ru}


\date{\today}

\newcommand{\0}[2]{{}_{#2}#1}
\newcommand{\9}[2]{{}^{#2}#1}
\newcommand{\eps}{\varepsilon}
\newcommand{\ii}{\boldsymbol{i}}

\newcommand{\CC}{\mathbb{C}}
\newcommand{\RR}{\mathbb{R}}
\newcommand{\NN}{\mathbb{N}}
\newcommand{\QQ}{\mathbb{Q}}
\newcommand{\ZZ}{\mathbb{Z}}
\newcommand{\PP}{\mathbb{P}}
\newcommand{\AAAA}{\mathbb{A}}

\newcommand{\AAA}{\mathsf{A}}
\newcommand{\BBB}{\mathsf{B}}
\newcommand{\CCC}{\mathsf{C}}
\newcommand{\GGG}{\mathsf{G}}

\newcommand{\aaa}{\mathfrak{a}}
\newcommand{\B}{\mathfrak{B}}
\newcommand{\bb}{\mathfrak{b}}
\newcommand{\g}{\mathfrak{g}}
\newcommand{\h}{\mathfrak{h}}
\newcommand{\llll}{\mathfrak{l}}
\newcommand{\p}{\mathfrak{p}}
\newcommand{\s}{\mathfrak{s}}
\newcommand{\ttt}{\mathfrak{t}}
\newcommand{\uu}{\mathfrak{u}}
\newcommand{\z}{\mathfrak{z}}
\newcommand{\gl}{\mathfrak{gl}}
\newcommand{\sgl}{\mathfrak{sl}}

\newcommand{\N}{\mathcal{N}}
\newcommand{\OO}{\mathcal{O}}
\newcommand{\Q}{\mathcal{Q}}
\newcommand{\SSS}{\mathcal{S}}
\newcommand{\U}{\mathcal{U}}
\newcommand{\V}{\mathcal{V}}

\newcommand{\Bmax}{\B_{\text{\upshape{max}}}}
\newcommand{\Bopen}{\B_{\text{\upshape{open}}}}
\newcommand{\Omax}{\OO_{\text{\upshape{max}}}}
\newcommand{\OOmax}{\widetilde\OO_{\text{\upshape{max}}}}
\newcommand{\NQ}{\Xi^*_{\QQ}}
\newcommand{\vlW}{W^{\sharp}}

\newcommand{\Aut}{\operatorname{Aut}}
\newcommand{\Hom}{\operatorname{Hom}}
\newcommand{\Img}{\operatorname{Im}}
\newcommand{\Ker}{\operatorname{Ker}}
\newcommand{\Spin}{\operatorname{Spin}}
\newcommand{\diag}{\operatorname{diag}}
\newcommand{\codim}{\operatorname{codim}}
\newcommand{\rk}{\operatorname{rk}}
\newcommand{\pr}{\text{\upshape{pr}}}

\newcommand{\bundle}[3]{{#1}\times^{#2}{#3}}
\newcommand{\Ru}[1]{#1_{\text{\upshape{u}}}}
\newcommand{\quot}{/\!\!/}

\newtheorem{theorem}{Theorem}
\newtheorem{proposition}[theorem]{Proposition}
\newtheorem{lemma}[theorem]{Lemma}
\newtheorem{corollary}[theorem]{Corollary}
\theoremstyle{definition}

\newtheorem{remark}[theorem]{Remark}
\newtheorem{question}[theorem]{Question}

\usepackage{color}

\newcounter{savecounter}

\begin{document}

\begin{abstract}
We identify the $G(\RR)$-orbits of the real locus $X(\RR)$ of any spherical complex variety $X$ defined over $\RR$ and homogeneous under a split connected reductive group $G$ defined also over $\RR$.
This is done by introducing some reflection operators on the set of real Borel orbits of $X(\RR)$.
We thus investigate the existence problem for an action of the Weyl group of $G$ on the set of  real Borel orbits of $X(\RR)$.
In particular, we determine the varieties $X$ for which these operators define an action of the very little Weyl group of $X$ on the set of open real Borel orbits of $X(\RR)$. This enables us to give a parametrization of the $G(\RR)$-orbits of $X(\RR)$ in terms of the orbits of this new action.
\end{abstract}

\maketitle

\section*{Introduction}

The celebrated Sylvester's law of inertia for 
real quadratic forms in $n$ variables can be interpreted as a parametrization of the $GL_n(\RR)$-orbits in the real locus of the set of 
complex quadratic forms equipped with the real structure defined by the complex conjugation of matrices.

The set of non-degenerate complex quadratic forms
is an instance of a complex symmetric space, a homogeneous space $X=G/H$ under a reductive complex algebraic group $G$ with $H$ being the fixed point set of an involution of $G$.
A description of the $G(\RR)$-orbits on the set $X(\RR)$ of real points of any symmetric space $X$ defined over $\RR$ was obtained in~\cite{CFT} by the authors via Galois cohomology, in the spirit of a previous work of Borel and Ji~\cite{BorelJi}.
The method employed in these works makes essential use of the polar decomposition for symmetric spaces.

In the present paper, we are concerned with the more general class of
spherical varieties, i.e., irreducible algebraic varieties acted on by a reductive group $G$ in such a way that they contain an open orbit for any Borel subgroup
of $G$. In particular, we give a parametrization of the $G(\RR)$-orbits in $X(\RR)$ for any spherical homogeneous $G$-variety $X$ defined over $\RR$ provided that $G$ is a split connected reductive group defined over $\RR$.
Contrary to the aforementioned work, the approach followed here does not use Galois cohomology, it is geometrical.

Our parametrization builds mainly on the study of the real locus $\Omax(\RR)$ of the open orbit in $X$ of a split Borel subgroup $B$ of $G$, with particular focus on its set of $B(\RR)$-orbits.
In the case of non-degenerate real quadratic forms in $n$ variables, the open orbits of the real upper triangular matrices are parameterized by $n$-tuples of signs $\pm$ and this identification is equivariant for the natural actions of the symmetric group of degree $n$ on the set of these orbits and on the set of such tuples.
This naturally leads us to also investigate the existence of an action of the Weyl group of $G$ on the set of $B(\RR)$-orbits of $X(\RR)$ for any spherical $X$.

In Section~\ref{sec:open-Borel-orb}, by considering the local structure of $X$, we obtain a parametrization of the $B(\RR)$-orbits of $\Omax(\RR)$ in terms of tuples of signs $\pm$
(Corollary~\ref{cor:paramet-borel-orb}).

In Section~\ref{sec:min&refl}, we study the orbits of any minimal parabolic subgroup of $G(\RR)$ on $X(\RR)$.
This enables us to determine when two $B(\RR)$-orbits of $\Omax(\RR)$ lie in one and the same $G(\RR)$-orbit of $X(\RR)$ (Proposition~\ref{prop:refl&G-orb}).
Based on this description, we introduce an action of the simple reflections  of the Weyl group $W$ of $G$ on the set of $B(\RR)$-orbits of $X(\RR)$.

Next in Section~\ref{sec:braid-rel}, we precisely identify when the simple reflections acting on the set
of open $B(\RR)$-orbits of $X(\RR)$ satisfy the braid relations of $W$.
Specifically, we show that the simple reflections define an action of the very little Weyl group of $X$ (for a definition see \S\ref{subsec:recalls}) on this set
only if $X$ belongs to a certain class of spherical homogeneous varieties
(see Theorem~\ref{thm:main-little-Weyl} and Remark~\ref{rmk:no-braid}).
Finally, for these alluded $X$, we get a parametrization of the $G(\RR)$-orbits on $X(\RR)$ in terms of the orbits of the very little Weyl group of $X$  on the set
of open $B(\RR)$-orbits of $X(\RR)$ (Corollary~\ref{cor:parametrization}); this result may be regarded as an analog of Sylvester's law of inertia.

In Section~\ref{sec:R-tangent}, we further investigate the existence problem for an action of the Weyl group of $G$ on the set of $B(\RR)$-orbits of $X(\RR)$.
Theorem~\ref{thm:real-Knop-action} is a partial answer to this problem:
we prove that the actions of the simple reflections of $W$ define an action of $W$ on the set of $B(\RR)$-orbits in $X(\RR)$ with maximal rank for the aforementioned peculiar $X$.
This is achieved 
by relating the real $B(\RR)$-orbits of maximal rank  
to the real cotangent bundle of $X(\RR)$ via conormal bundles in such a way that the action of $W$ becomes manifest.
This approach has been inspired by the work of Knop~\cite{knop:action},
where an action of $W$ on the set of $B$-orbits of $X$ is constructed.
The part of this work with which we are concerned is reviewed in Section~\ref{sec:Weyl-C}.

In Section~\ref{sec:examples} we conclude this work by discussing some examples.

\subsection*{Acknowledgment}

This work was partially supported by the SFB/TRR 191 ``Symplectic Structures in Geometry, Algebra and Dynamics'' of the Deutsche Forschungsgemeinschaft.
The second author acknowledges partial support from the Russian Foundation for Basic Research, project number 20-01-00091.

\section{Notation and basic material}
\label{sec:basic}

The ground field is the field $\CC$ of complex numbers.
Complex algebraic varieties and groups are identified with the sets of their complex points.
For an algebraic variety $Y$ defined over $\RR$,  we let $Y(\RR)$ be its set of real points and we denote the complex conjugation map defining the real structure on $Y$ by $y\mapsto\bar{y}$ unless otherwise specified.

Algebraic groups and Lie groups are written with capital Latin letters and their Lie algebras with respective lowercase German letters.
The identity component of an algebraic (Lie) group $H$ is denoted by $H^\circ$ and the center of $H$ is denoted by $Z(H)$. The unipotent radical of $H$ is denoted by $\Ru{H}$.

Throughout, $G$ is a split connected reductive algebraic group defined over $\RR$. The complex conjugation on $G$ is denoted by
$$
\sigma:G\longrightarrow G, \quad\sigma(g)=\bar{g}.
$$
Given any subset $H\subset G$, we denote the set of elements of $H$ fixed by $\sigma$ with $H^\sigma$.

We fix a split maximal torus $T\subset G$ and a Borel subgroup $B\subset G$ containing $T$. Both subgroups are defined over $\RR$.
Let $U$ be the maximal unipotent subgroup of $B$ and $W=N_G(T)/T$ be the Weyl group of $G$.
Let $B^-$ denote the opposite Borel subgroup of $G$, that is $B\cap B^-=T$, and $U^-$ be the maximal unipotent subgroup of $B^-$. The root subgroup corresponding to a root $\alpha$ of $G$ with respect to $T$ is denoted by $U_\alpha$; it consists of the elements $u_{\alpha}(s)=\exp se_{\alpha}$, where $e_{\alpha}$ is a generator of the respective root subspace $\g_{\alpha}\subset\g$.

Recall that an irreducible variety $X$ equipped with an action of $G$ is called \emph{spherical} if any Borel subgroup of $G$ has an open dense orbit on $X$.
As originally proved by Brion and Vinberg independently, this property is equivalent to $X$ having finitely many orbits for any Borel subgroup of $G$
(see e.g.~\cite[2.6]{knop:action}).

Henceforth, we let $X$ be a spherical homogeneous $G$-variety defined over $\RR$ and $\Omax$ be its open $B$-orbit.
The complex conjugation on $X$ is denoted by
$$
\mu:X\longrightarrow X,\quad \mu(x)=\bar{x}.
$$
We assume $X(\RR)\ne\varnothing$ and we fix a base point $x_0\in X(\RR)$.
We may choose $x_0$ in the orbit $\Omax$.
Let $H$ be the stabilizer of $x_0$ in $G$. Note that $H$ is a subgroup of $G$ defined over $\RR$.

Conversely, suppose $X=G/H$ is spherical and $H$ is defined over $\RR$.
Then $X$ admits a unique real structure $\mu$ compatible with the action of $G$ (i.e., $\mu(gx)=\sigma(g)\mu(x)$ for all $(g,x)\in G\times X$) and such that $x_0$ is real w.r.t.~$\mu$.
Namely, $\mu$ is the real structure given by
$$
x=gx_0\longmapsto\mu(x)=\sigma(g)x_0\quad\text{for any $g\in G$}.
$$

The automorphism group of $X$ acts on the set of real structures on $X$ by conjugation. Note that the group of $G$-equivariant automorphisms of $X$ is isomorphic to $N_G(H)/H$ with $N_G(H)$ being the normalizer of $H$ in $G$.

\section{Open real Borel orbits in the real locus}\label{sec:open-Borel-orb}

The real locus $X^{\mu}=X(\RR)$ is a smooth real analytic manifold of real dimension equal to the (complex) dimension of $X$: it is a totally real submanifold of $X$.
It decomposes into a disjoint union of finitely many orbits of $G(\RR)=G^{\sigma}$ (see \cite[Chap.~III, 4.2 and 4.4]{Galois}), which are Zariski dense in $X$ as well as open and closed in the classical Hausdorff topology of $X^{\mu}$.
This implies readily

\begin{proposition}
Each $G^{\sigma}$-orbit on $X^{\mu}$ intersects $\Omax$ in finitely many $B^{\sigma}$-orbits.
\end{proposition}

Consider the following parabolic subgroup of $G$:
$$
P=\{g\in G\mid g\Omax=\Omax\}.
$$
Note that $P$ contains $B$. Let $P=\Ru{P}\rtimes L$ denote the Levi decomposition of $P$ with $L\supseteq T$.

\begin{theorem}\label{thm:loc-str}
There exists a closed $L$-stable subvariety $Z\subset\Omax$ such that
\begin{itemize}
  \item $\Omax\simeq\bundle{P}{L}{Z}\simeq\Ru{P}\times Z$ (as $P$-varieties) and
  \item $L$ acts on $Z$ transitively with its semisimple part acting trivially.
\end{itemize}
Furthermore, $Z$ may be chosen to be defined over $\RR$.
\end{theorem}

We call such a $Z$ a \emph{slice} of $X$. Throughout this paper, we take $Z$ to be defined over $\RR$.

Theorem~\ref{thm:loc-str} is known as (the generic version of) the \emph{Local Structure Theorem}.
There are several known constructions of a slice, due to Brion--Luna--Vust \cite{BLV}, Knop \cite{inv.mot} et al.
For us, Knop's construction, which we will review later (see Proposition~\ref{prop:flats}), will be relevant; in particular, it implies readily that the slice may be chosen to be defined over~$\RR$.

\begin{corollary}\label{cor:struct-loc}
Let $Z$ be a slice of $X$ defined over $\RR$. Then
\begin{equation*}
\Omax^{\mu}\simeq\Ru{P}^{\sigma}\times Z^{\mu}.
\end{equation*}
\end{corollary}

Let $L_0$ denote the kernel of the $L$-action on $Z$. Note that $L_0$ contains the semisimple part $[L,L]$ of $L$. Let $T_0=T\cap L_0$. Then
$$
A:=L/L_0=T/T_0
$$
is a torus acting freely and transitively on $Z$.
By choosing the base point $x_0$ in the real locus $Z^{\mu}$, we may identify $Z$ with $A$ via the orbit map.

Let $\0T2$ (resp.\ $\0A2$) be the subgroup of $T$ (resp.\ $A$) consisting of the elements of order $\le2$.
While identifying $Z$ with $A$, we denote the finite subset of $Z^{\mu}$ corresponding to~$\0A2$ by $\0Z2$.

Let us choose compatible bases $\eps_1,\dots,\eps_n$ and $m_1\eps_1,\dots,m_r\eps_r$ ($m_i\in\NN$, for $i=1,...,r$) of the character lattices of $T$ and $A$, the latter being a sublattice of the former one.
In the coordinates provided by these bases, the quotient map $T\to A$, identified with the orbit map for the $T$-action on $Z$, is given by the formula
$$
(t_1,\dots,t_n)\longmapsto(t_1^{m_1},\dots,t_r^{m_r}).
$$

\begin{proposition}\label{prop:paramet-borel-orb}
The orbits of $T^{\sigma}$ on $Z^{\mu}$ are parameterized by the tuples of $\pm$ which are the signs of the coordinates on $Z^{\mu}$ corresponding to even $m_i$'s. In particular, there are $2^p$ orbits of $T^{\sigma}$ on $Z^{\mu}$,
with $p$ being the number of even $m_i$'s.

Moreover, each $T^{\sigma}$-orbit intersects $\0Z2$ in one $\0T2$-orbit.
\end{proposition}

\begin{proof}
By considering the bases chosen above,
we note that the $T^{\sigma}$-action on $Z^{\mu}$ varies the absolute values of all coordinates arbitrarily and alters the signs of the coordinates corresponding to odd~$m_i$.
The proposition follows readily.
\end{proof}

Proposition~\ref{prop:paramet-borel-orb} improves \cite[Prop. 5.3]{ACF}, where an upper bound for the number of $(T^\sigma)^\circ$-orbits of $Z^\mu$ is given while considering some peculiar $X$.

\begin{corollary}\label{cor:paramet-borel-orb}
The $B^\sigma$-orbits on $\Omax^\mu$ are in bijection with the $\0T2$-orbits on $\0Z2$.
\end{corollary}

\begin{proof}
By Corollary~\ref{cor:struct-loc}, any $B^{\sigma}$-orbit on $\Omax^{\mu}$ intersects $Z^{\mu}$ exactly in one $T^{\sigma}$-orbit.
This together with Proposition~\ref{prop:paramet-borel-orb} implies the corollary.
\end{proof}


\section{Simple reflection operators on real orbits}
\label{sec:min&refl}

Let $\B(X)$ and $\B(X^\mu)$ denote the sets of $B$-orbits on $X$ and $B^\sigma$-orbits on $X^\mu$ respectively.

In this section, after recalling the action of any simple reflection of $W$ on  $\B(X)$ as defined by Knop in~\cite{knop:action}, we introduce an action of any such reflection on the set $\B(X^\mu)$.

Consider the set of simple roots given by $B$ and $T$.
For a simple root $\alpha$, let $s_\alpha\in W$ be the corresponding reflection and let $P_{\alpha}$ denote the corresponding minimal parabolic subgroup of $G$, that is
$$
P_{\alpha}= B\cup Bn_\alpha B,
$$
with $n_{\alpha}\in N_G(T)$ being any representative of $s_{\alpha}$.

\subsection{Complex case}
\label{subsec:min&refl-C}

Take $\OO$ in $\B(X)$  and consider its $P_{\alpha}$-span $P_{\alpha}\OO=P_{\alpha}x$.
Without loss of generality, we may assume that $\OO$ is open in $P_{\alpha}\OO$.

Let $(P_\alpha)_x$ be the stabilizer of $x$ in $P_\alpha$.
Note that there is a bijective correspondence between 
the set of $B$-orbits in $P_{\alpha}\OO$ and the set of $(P_{\alpha})_x$-orbits in $P_{\alpha}/B\simeq\PP^1$.
Explicitly, to a $B$-orbit $\OO'=Bx'$ with $x'=gx$ and $g\in P_{\alpha}$, there corresponds
the $(P_{\alpha})_x$-orbit of $y=g^{-1}(\infty)$,
where $\infty$ denotes the point $[1:0]\in\PP^1$ identified with the base point of $P_{\alpha}/B$.
This correspondence preserves codimensions and the closure ordering.

Let $Q_x$ denote the image of $(P_{\alpha})_x$ under the homomorphism $P_{\alpha}\to\Aut\PP^1=PGL_2$; it is the quotient of $(P_{\alpha})_x$ which acts effectively on~$\PP^1$.
Up to conjugation, one of the following cases may occur.
\begin{list}{}{}
  \item[(P)] $Q_x=PGL_2$ acts transitively on $\PP^1$. Then $P_{\alpha}\OO=\OO$ is a single $B$-orbit.
  \item[(U)] $U(PGL_2)\subseteq Q_x\subseteq B(PGL_2)$, where
       \begin{equation*}
         U(PGL_2)=\left\{
         \begin{bmatrix}
           1 & * \\
           0 & 1 \\
         \end{bmatrix}
         \right\}\subset PGL_2\quad\text{and}\quad
         B(PGL_2)=\left\{
         \begin{bmatrix}
           * & * \\
           0 & * \\
         \end{bmatrix}
         \right\}\subset PGL_2.
       \end{equation*}
       Here $Q_x$ acts transitively on $\AAAA^1=\PP^1\setminus\{\infty\}$.
       Then $P_{\alpha}\OO=\OO\cup\OO'$, where $\OO'$ is a $B$-orbit of codimension~1.
   \item[(T)] $Q_x=T(PGL_2)$, where
       $$
         T(PGL_2)=\left\{
         \begin{bmatrix}
           * & 0 \\
           0 & * \\
         \end{bmatrix}
         \right\}\subset PGL_2.
       $$
       Here $Q_x$ acts transitively on $\PP^1\setminus\{0,\infty\}$ and fixes both $0$ and~$\infty$.
       Then $P_{\alpha}\OO=\OO\cup\OO_1\cup\OO_2$, where $\OO_1$ and $\OO_2$ are $B$-orbits of codimension~1.
   \item[(N)] $Q_x=N(PGL_2)$, where
       \begin{align*}
         N(PGL_2)&=\left\{
         \begin{bmatrix}
           * & 0 \\
           0 & * \\
         \end{bmatrix},
         \begin{bmatrix}
           0 & * \\
           * & 0 \\
         \end{bmatrix}
         \right\}\subset PGL_2.
       \end{align*}
       Here $Q_x$ acts transitively on $\PP^1\setminus\{0,\infty\}$ and permutes $0$ and $\infty$. Then $P_{\alpha}\OO=\OO\cup\OO_0$, where $\OO_0$ is a $B$-orbit of codimension~1.
\end{list}

Given $\OO\in\B(X)$, the \emph{type} of the simple root $\alpha$ (with respect to $\OO$) will refer to the different situations above.

\begin{remark}
For the sake of convenience,
we will replace the group $Q_x$, in types (T) and (N), by its conjugate $PSO_{1,1}(\CC)$ and $PO_{1,1}(\CC)$, respectively.
Accordingly, the points $0=[0:1]$ and $\infty=[1:0]$ will be replaced by $1=[1:1]$ and $-1=[-1:1]$ in the above description.
\end{remark}

Following~{\cite[\S 4]{knop:action}}, we define an action of the simple reflection $s_{\alpha}$ on the set $\B(X)$ by setting, accordingly to the type of $\alpha$ with respect to $\OO\in\B(X)$:
\begin{list}{}{}
  \item[(P)] $\9\OO{s_\alpha}=\OO$;
  \item[(U)] $\9\OO{s_\alpha}=\OO'$ and $\9{\OO'}{s_\alpha}=\OO$;
  \item[(T)] $\9\OO{s_\alpha}=\OO$, $\9{\OO_1}{s_\alpha}=\OO_2$, and $\9{\OO_2}{s_\alpha}=\OO_1$;
  \item[(N)] $\9\OO{s_\alpha}=\OO$ and $\9{\OO_0}{s_\alpha}=\OO_0$.
  \end{list}
Here and below, we denote the action of a simple reflection by putting the respective superscript on the left.

In loc.~cit., Knop proves that these reflection operators satisfy the braid relations.

\begin{theorem}[Knop]\label{thm:knop-action}
The simple reflection operators introduced above define an action of the Weyl group $W$ on the set $\B(X)$ of Borel orbits.
\end{theorem}

\begin{remark}\label{rmk:rk-reduction}
Sometimes it is convenient to study the reflection operators by induction on the semisimple rank of $G$. The following observation makes it possible.

Let $S$ be any set of simple roots and $P_S\supseteq B$ be the standard parabolic subgroup of $G$ defined by $S$, with the Levi decomposition $P_S=\Ru{(P_S)}\rtimes L_S$ such that $L_S\supseteq T$. Then $B_S=B\cap L_S$ is a Borel subgroup in $L_S$.

Let $Y=P_Sx$ be any $P_S$-orbit in $X$. Consider the quotient map $\pi: Y\to V=Y/\Ru{(P_S)}=L_Sz$, where $z$ is the image of $x$ under this map. The stabilizer $(L_S)_z$ is the image of $(P_S)_x$ under the homomorphism $P_S\to P_S/\Ru{(P_S)}\simeq L_S$. Observe that the $B$-orbits in $Y$ are the preimages under $\pi$ of the $B_S$-orbits in $V$ and this correspondence preserves the configuration types of Borel orbits vs.\ simple roots in $S$.
\end{remark}

\subsection{Symmetric spaces}\label{subsec:symm}

We give a characterization of the types of simple roots with respect to certain $B$-orbits (namely those of maximal rank, see~\S\ref{subsec:rank} in case $X$ is a symmetric space. This will be of importance later; see 
Theorem~\ref{thm:real-Knop-action} and its proof.

Before stating this characterization, we recall a few well-known facts about the structure theory of symmetric spaces; see e.g.\ \cite[\S26]{tim} for a review.

Let $X=G/H$ be a symmetric space, i.e., $H^{\circ}=(G^{\theta})^{\circ}$ for some non-identical involution $\theta$ of $G$.
As proved by Vust~\cite{vust}, symmetric $G$-spaces are spherical $G$-varieties.

Let $T_1$ be a maximal $\theta$-split torus, namely a torus on which $\theta$ acts as the inversion and which is maximal for this property, and
$L$ be the centralizer of $T_1$ in $G$. 
Take a maximal torus $T_0\subset L^{\theta}$.
Then $L=T_1\cdot L^{\theta}$ and $T=T_0\cdot T_1$ is a maximal torus in $G$ which is $\theta$-stable.
Let us choose a parabolic subgroup $P$ of $G$ having $L$ as a Levi subgroup and such that $\theta(P)=P^-$
as well as a Borel subgroup $B$ of $G$ such that $T\subset B\subset P$.

The root system of $G$ with respect to $T$ is thus acted on by $\theta$ and the roots $\alpha$ of $(G,B,T)$ are subdivided into \emph{real} ($\theta(\alpha)=-\alpha$), \emph{imaginary} ($\theta(\alpha)=\alpha$),
and \emph{complex} ($\theta(\alpha)\ne\pm\alpha$) roots.
Real and complex positive roots are mapped by $\theta$ to negative roots.

Generators $e_{\alpha}$ of the root subspaces $\g_{\alpha}$ can be chosen in such a way that $\theta(e_{\alpha})=e_{\theta(\alpha)}$. The Lie algebra of $H$ is
$$
\h=\ttt_0\oplus\bigoplus_{\alpha=\theta(\alpha)}\g_{\alpha}\oplus
\bigoplus_{\alpha\ne\theta(\alpha)}\h_{\alpha,\theta(\alpha)}
$$
with $\h_{\alpha,\theta(\alpha)}$ being generated by $e_{\alpha}+e_{\theta(\alpha)}$ and  being ``diagonally'' embedded in $\g_{\alpha}\oplus\g_{\theta(\alpha)}$.
From the Iwasawa decomposition $\g=\Ru\p\oplus\ttt_1\oplus\h$, ones sees
that the base point $x_0$ (with stabilizer $H$) lies in $\Omax$ and that
$P$ is the maximal parabolic subgroup of $G$ stabilizing $\Omax$.

Let $n\in N_G(T)$ and $x=nx_0$.
Given a simple root $\alpha$ of $G$, we have $(P_{\alpha})_x=P_{\alpha}\cap nHn^{-1}$ and $nHn^{-1}$ is a symmetric subgroup of $G$ corresponding to the involution $\theta_x=\iota_n\theta\iota_n^{-1}$, with $\iota_n$ being the conjugation by $n$.
The torus $T$ is thus preserved by $\theta_x$ and $nT_1n^{-1}$ is a maximal $\theta_x$-split torus hence we may consider root types \emph{with respect to} $\theta_x$.

\begin{proposition}\label{prop:root-symm}
Let $X$ be a symmetric space. Take $n\in N_G(T)$ and let $x=nx_0$.
The types of simple roots with respect to the orbit $\OO=Bx$ are determined as follows.
\begin{enumerate}
\item
If $\alpha$ is a complex root w.r.t.\ $\theta_x$, then the root $\alpha$ is of type (U).
\item
If $\alpha$ is a real root w.r.t.\ $\theta_x$, then  the root $\alpha$ is of type (T) or (N).
Moreover,  $(\llll_{\alpha})_z=\llll_{\alpha}^{\theta_x}= [\llll_{\alpha},\llll_{\alpha}]^{\theta_x}\oplus(\ttt^{\alpha})^{\theta_x}$, where $\ttt^{\alpha}=\z(\llll_{\alpha})$.
\item
If $\alpha$ is an imaginary root w.r.t.\ $\theta_x$, then the root $\alpha$ is of type (P).
\end{enumerate}
\end{proposition}

\begin{proof}
Recall Remark~\ref{rmk:rk-reduction} and the notation set up therein with $S=\{\alpha\}$.

(1) Note that in this case $(L_{\alpha})_z$ contains one of the root subgroups $U_{\alpha}$ or $U_{-\alpha}$, depending on which of $U_{\pm\theta_x(\alpha)}$ is contained in $\Ru{(P_{\alpha})}$.

(2) Note that  $L_{\alpha}$ is $\theta_x$-stable, $(L_{\alpha})_z$ equals $L_{\alpha}^{\theta_x}$ up to connected components, and
$\theta_x$ preserves the derived subgroup $[L_{\alpha},L_{\alpha}]$ of $L_\alpha$ and the connected center $T^{\alpha}$ of $L_{\alpha}$ as well.

(3) Clearly, $L_{\alpha}$ is $\theta_x$-stable and $[L_{\alpha},L_{\alpha}]$ is pointwise fixed by $\theta_x$
hence $(L_{\alpha})_z\supset[L_{\alpha},L_{\alpha}]$.
\end{proof}

\subsection{Real case}
\label{subsec:min&refl-R}

To define an action of the simple reflections on the set $\B(X^\mu)$, we start by describing the configuration of the $B^{\sigma}$-orbits in $P_{\alpha}^{\sigma}\OO$ for a given $\OO$ in $\B(X^\mu)$.

Take $\OO$ in $\B(X^\mu)$ and consider its $P_{\alpha}^{\sigma}$-span $P_{\alpha}^{\sigma}\OO=P_{\alpha}^{\sigma}x$.
Without loss of generality, we assume that $\OO$ is open in $P_{\alpha}^{\sigma}\OO$.

Similarly as in the complex case, there is a bijective correspondence between the $B^{\sigma}$-orbits in $P_{\alpha}^{\sigma}\OO$ and the $(P_{\alpha})^{\sigma}_x$-orbits in $P_{\alpha}^{\sigma}/B^{\sigma}\simeq\PP^1(\RR)$.
The stabilizer $(P_{\alpha})^{\sigma}_x$ acts on $\PP^1(\RR)$ via its image $R_x$ under the homomorphism
$P_\alpha^\sigma\rightarrow PGL_2(\RR)$.

Note that $R_x$ is not just the real locus of the image $Q_x$ of $(P_\alpha)_x$ in $PGL_2$: the Lie groups $R_x$ and $Q_x(\RR)$ coincide only up to connected components.
More precisely, $(P_{\alpha})_x$ projects onto an algebraic subgroup of the Levi subgroup $L_{\alpha}$ of $P_\alpha$ containing $T$ whereas $R_x$ is the image of the real locus of this subgroup of $L_\alpha$.
Since $L_{\alpha}$ is isomorphic to the direct product of a torus with $SL_2$, $GL_2$ or $PGL_2$,
we may consider the real locus of an algebraic subgroup in $SL_2$, $GL_2$ or $PGL_2$ instead.

The types (P), (U), (T) and (N) may subdivide into ``real forms'', depending on the conjugacy class of $R_x$.
Specifically, in types (P) and (U), there is no subdivision, namely the configuration of $B^{\sigma}$-orbits is the same as in the complex case, whereas the types (T) and (N) both subdivide into three subcases.

\begin{list}{}{}
  \item[(T0)] $R_x=PSO_2(\RR)$ acts on $\PP^1(\RR)$ transitively. Then $P_{\alpha}^{\sigma}\OO=\OO$ is a single $B^{\sigma}$-orbit.
  \item[(T1)] $R_x=PSO_{1,1}(\RR)\times\langle r\rangle_2$, with $\langle r\rangle_2$ being the group (of order $2$) generated by
       \begin{equation*}
         r=\begin{bmatrix}
             0 & 1 \\
             1 & 0 \\
           \end{bmatrix}.
       \end{equation*}
       Here $R_x$ acts transitively on $\PP^1(\RR)\setminus\{\pm1\}$ and fixes $1$ and~$-1$.
       Therefore $P_{\alpha}^{\sigma}\OO=\OO\cup\OO_1\cup\OO_2$, where $\OO_1$ and $\OO_2$ are $B^{\sigma}$-or\-bits of codimension~1.
  \item[(T2)] $R_x=PSO_{1,1}(\RR)$ acts on $\PP^1(\RR)$ with four orbits: two fixed points $1$ and $-1$ and two open semicircles. Hence $P_{\alpha}^{\sigma}\OO=\OO\cup\OO'\cup\OO_1\cup\OO_2$, where the $B^{\sigma}$-orbits $\OO$ and $\OO'$ are open
  whereas $\OO_1$ and $\OO_2$ have codimension~1.
  \item[(N0)] $R_x=PO_2(\RR)$ acts transitively on $\PP^1(\RR)$. Then $P_{\alpha}^{\sigma}\OO=\OO$ is a single $B^{\sigma}$-orbit.
  \item[(N1)] $R_x=P(S)O_{1,1}(\RR)\rtimes\langle n\rangle_2$, where
       \begin{equation*}
         n=\begin{bmatrix}
             0 &-1 \\
             1 & 0 \\
           \end{bmatrix}.
       \end{equation*}
       Here $R_x$ acts transitively on $\PP^1(\RR)\setminus\{\pm1\}$ and permutes $1$ and~$-1$.
       Therefore $P_{\alpha}^{\sigma}\OO=\OO\cup\OO_0$, where $\OO_0$ is a $B^{\sigma}$-orbit of codimension~1.
  \item[(N2)] $R_x=PO_{1,1}(\RR)$ acts on $\PP^1(\RR)$ with three orbits: $\{\pm1\}$ and two open semicircles. Therefore $P_{\alpha}^{\sigma}\OO=\OO\cup\OO'\cup\OO_0$, where the $B^{\sigma}$-orbits $\OO$ and $\OO'$ are open and $\OO_0$ has codimension~1.
\end{list}

The action of the simple reflection $s_{\alpha}$ on the set $\B(X^\mu)$ is defined as in the complex case with the following addition: the open orbits $\OO$ and $\OO'$ in (T2) and (N2) are interchanged. 
Specifically, $s_\alpha$ acts on the set of $B^{\sigma}$-orbits in $P_{\alpha}^{\sigma}\OO$ as follows.

\begin{list}{}{}
  \item[(P)]  $\9\OO{s_\alpha}=\OO$;
  \item[(U)]  $\9\OO{s_\alpha}=\OO'$ and $\9{\OO'}{s_\alpha}=\OO$;
  \item[(T0)] $\9\OO{s_\alpha}=\OO$;
  \item[(T1)] $\9\OO{s_\alpha}=\OO$, $\9{\OO_1}{s_\alpha}=\OO_2$, and $\9{\OO_2}{s_\alpha}=\OO_1$;
  \item[(T2)] $\9\OO{s_\alpha}=\OO'$, $\9{\OO'}{s_\alpha}=\OO$, $\9{\OO_1}{s_\alpha}=\OO_2$, and $\9{\OO_2}{s_\alpha}=\OO_1$;
  \item[(N0)] $\9\OO{s_\alpha}=\OO$;
  \item[(N1)] $\9\OO{s_\alpha}=\OO$ and $\9{\OO_0}{s_\alpha}=\OO_0$;
  \item[(N2)] $\9\OO{s_\alpha}=\OO'$, $\9{\OO'}{s_\alpha}=\OO$, and $\9{\OO_0}{s_\alpha}=\OO_0$.
  \end{list}

In the next sections, we will determine when the simple reflections acting  on the set of $B^{\sigma}$-orbits in the real loci of $B$-orbits of maximal rank
define an action of the (very little) Weyl group.

\begin{question}\label{que:W-act-real}
When do the reflection operators define an action of $W$ on the whole set $\B(X^{\mu})$ of $B^{\sigma}$-orbits in $X^{\mu}$?
\end{question}

To address this question and for other purposes, it is helpful to extend Remark~\ref{rmk:rk-reduction} over $\RR$.

\begin{remark}\label{rmk:rk-red-real}
Observe that for any set $S$ of simple roots the standard parabolic subgroup $P_S$, its unipotent radical, and the standard Levi subgroup $L_S$ are defined over $\RR$. Also, for any $x\in X^{\mu}$, the varieties $Y=P_Sx$, $V=Y/\Ru{(P_S)}$, and the quotient map $\pi:Y\to V$ are defined over $\RR$.
Moreover, $P_S^{\sigma}x$ is open and closed in $Y^{\mu}$ and the $B^{\sigma}$-orbits in $P_S^{\sigma}x$ are the preimages under $\pi|_{Y^{\mu}}$ of the $B_S^{\sigma}$-orbits in $L_S^{\sigma}z\subseteq V^{\mu}$, where $z=\pi(x)$.

Indeed, such a preimage is a single $B^{\sigma}$-orbit, because $\Ru{(P_S)}^{\sigma}$ acts transitively on the real loci of the fibers of $\pi$, which follows e.g.\ from the vanishing of the Galois cohomology of unipotent groups, see \cite[I.5.4, Cor.\,1 and III.2.1, Prop.\,6]{Galois}.
Again, the correspondence between the $B^{\sigma}$-orbits in $P_S^{\sigma}x$ and the $B_S^{\sigma}$-orbits in $L_S^{\sigma}z$ preserves the configuration types of real Borel orbits vs.\ simple roots in $S$.
\end{remark}

In view of the above discussion on the $P_{\alpha}^{\sigma}$-spans of  $B^{\sigma}$-orbits of $X^\mu$, the next proposition immediately follows from the fact that $G^{\sigma}$ is generated by its minimal parabolic subgroups.

\begin{proposition}\label{prop:refl&G-orb}
Two open $B^{\sigma}$-orbits $\OO$ and $\OO'$ of $X^{\mu}$ lie in one and the same $G^{\sigma}$-orbit if and only if $\OO'$ is obtained from $\OO$ by a sequence of reflection operators of type $(T)$ or $(N)$ with respect to $\Omax$.
\end{proposition}

\section{Action of the very little Weyl group}\label{sec:braid-rel}

In this section, we start by collecting a few notions from the theory of spherical varieties, after which
we determine when the actions of the simple reflections introduced in Section~\ref{sec:min&refl} define an action of the very little Weyl group of the given spherical homogeneous space $X$.

\subsection{Recalls on spherical varieties}\label{subsec:recalls}

For the recalls freely made in this subsection, the reader may consult~\cite{tim} for a survey.

Let $\Xi(X)$ denote the weight lattice of $X$, that is the set of weights of the $B$-eigenfunctions in the function field $\CC(X)$ of $X$.
With the notation of \S\ref{sec:open-Borel-orb}, $\Xi(X)$ is nothing but the character lattice of $A$.
The rank of the lattice $\Xi(X)$ is called \emph{the rank} of $X$.

Set $\NQ(X)=\Hom_{\ZZ}(\Xi(X),\QQ)$.

Let $\V(X)$ be the set of $G$-invariant $\QQ$-valued discrete valuations of $\CC(X)$. Each element $v$ of $\V(X)$ yields an element $\rho_v$ of $\NQ(X)$ given by the map
$\chi \mapsto v(f)$, where $f$ denotes a $B$-eigenfunction of weight $\chi$ (uniquely determined up to a scalar in $\CC^\times$). The map $v\mapsto \rho_v$ is injective; this map identifies $\V(X)$ with a convex co-simplicial cone (still denoted $\V(X)$) in $\NQ(X)$.
There thus exists a linearly independent set $\Sigma_X$ of primitive elements in the lattice $\Xi(X)$ such that
$$
\V(X)=\left\{\rho\in \NQ(X): \left\langle\rho,\gamma\right\rangle\leq 0, \, \forall\gamma\in\Sigma_X\right\}.
$$
The set $\Sigma_X$ is called the set of \emph{spherical roots of $X$}.
It is a basis of a certain root system in~$\Xi(X)$.

Recall that $L$ denotes the Levi subgroup of the parabolic subgroup $P$ of $G$ stabilizing the open $B$-orbit $\Omax$ of $X$ such that $L\supset T$.

\begin{proposition}[{\cite[Prop.\,2.3, Cor.\,5.3]{knop:loc}}]
\label{prop:types-spherroots}
The types of simple roots w.r.t.\ $\Omax$ are characterized by the spherical roots of $X$ as follows.
\begin{enumerate}
\item
  $\alpha$ is of type $(P)$ if and only if $\alpha$ is a root of $L$.
  \item $\alpha$ is of type $(U)$ if and only if $\alpha\notin\Sigma_X\cup\frac12\Sigma_X$.
  \item $\alpha$ is of type $(T)$ if and only if $\alpha\in\Sigma_X$.
  \item $\alpha$ is of type $(N)$ if and only if $2\alpha\in\Sigma_X$.
\end{enumerate}
\end{proposition}

The valuation cone $\V(X)$ of $X$ is a fundamental domain for a finite crystallographic reflection group, \emph{the little Weyl group of $X$}, which is the Weyl group of the root system generated by $\Sigma_X$.

Let $W_{(X)}$ be the stabilizer of the orbit $\Omax$ with respect to Knop's action of the Weyl group $W$ (Theorem~\ref{thm:knop-action}). Note that the Weyl group $W_L$ of $L$ is contained in $W_{(X)}$.
The group $W_{(X)}$ is the semi-direct product of $W_L$ with the group
$$
W_X=\left\{w\in W: w(\Omax)=\Omax\text{ and } w(S_X)=S_X\right\},
$$
where $S_X$ is the set of simple roots of $L$.
According to Theorem 6.2 in~\cite{knop:action}, $W_X$ is the little Weyl group of $X$.

We call the subgroup of $W$ generated by the simple reflections contained in $W_X$ \emph{the very little Weyl group of $X$} and denote it by $\vlW_X$.

A set of generators for $W_X$ was first given by Brion (see e.g. Theorem 7 in~\cite{brion}). Thanks to this result, we can state the following.

\begin{theorem}\label{thm:brion-little-Weyl}
The very little Weyl group $\vlW_X$ of $X$ is generated by the simple reflections associated to simple roots of $G$ belonging to $\Sigma_X$ or $\frac{1}{2}\Sigma_X$.
\end{theorem}

\begin{corollary}
The group $\vlW_X$ is generated by the simple reflections associated to simple roots of $G$ of type $(T)$ or $(N)$ w.r.t.\ $\Omax$.
\end{corollary}

\begin{proof}
This follows from Theorem~\ref{thm:brion-little-Weyl} and Proposition~\ref{prop:types-spherroots}.
\end{proof}

\subsection{Braid relations}\label{subsec:braid}

Let $\Bopen(X^{\mu})$ denote the set of $B^{\sigma}$-orbits in $\Omax^{\mu}$, i.e., open $B^{\sigma}$-orbits in $X^{\mu}$.

\begin{theorem}\label{thm:main-little-Weyl}
Suppose $X$ has no adjacent simple spherical roots of same length.
Then the reflection operators associated to simple roots in $\Sigma_X\cup\frac{1}{2}\Sigma_X$ define an action of $\vlW_X$ on $\Bopen(X^{\mu})$. 
\end{theorem}

\begin{corollary}\label{cor:parametrization}
Suppose $X$ has no adjacent simple spherical roots of same length.
Then the $G(\RR)$-orbits of $X(\RR)$ are parameterized by the orbits of $\vlW_X$ acting on the set of open $B(\RR)$-orbits in $X(\RR)$.
\end{corollary}

\begin{proof}
This follows readily from Proposition~\ref{prop:refl&G-orb} and Theorem~\ref{thm:main-little-Weyl}.
\end{proof}

\begin{proof}[Proof of Theorem~\ref{thm:main-little-Weyl}]
We start with some reductions.

The spherical homogeneous space $X=G/H$ admits a $G$-equivariant finite cover by another spherical homogeneous space $\widetilde{X}=G/H^{\circ}$, which carries on a canonical real structure such that the fiber over $x_0$ contains a real point $\tilde{x}_0$; see~\S\ref{sec:basic}. The orbit $G^{\sigma}x_0=G^{\sigma}/H^{\sigma}$ is covered by $G^{\sigma}\tilde{x}_0=G^{\sigma}/(H^{\circ})^{\sigma}$ and the $B^{\sigma}$-orbits in $G^{\sigma}/H^{\sigma}$ are the images of the $B^{\sigma}$-orbits in $G^{\sigma}/(H^{\circ})^{\sigma}$. This correspondence between $B^{\sigma}$-orbits is compatible with the reflection operators. Since the simple reflections operate on $B^{\sigma}$-orbits within a given $G^{\sigma}$-orbit, we reduce the question to the case of a connected $H$.

The reductive group $G$ admits a finite cover by $\widetilde{G}=G_0\times C$ defined over $\RR$, where $G_0$ is the simply connected covering group of the semisimple part $[G,G]$ of $G$ and $C$ is the connected center of $G$. The preimage $\widetilde{B}$ of $B$ is a Borel subgroup of $\widetilde{G}$ defined over $\RR$ and $\widetilde{B}=B_0\times C$, where $B_0$ is a Borel subgroup of $G_0$.

Replacing the acting group $G$ with $\widetilde{G}$ may have the effect of splitting the real orbits in $X^{\mu}$ into smaller orbits, because the homomorphisms $\widetilde{G}^{\sigma}\to G^{\sigma}$ and $\widetilde{B}^{\sigma}\to B^{\sigma}$ are not necessarily surjective. Still, the action of simple reflections on $B^{\sigma}$-orbits is compatible with the action on smaller $\widetilde{B}^{\sigma}$-orbits. Therefore we may assume without any loss of generality that $G$ is \emph{of simply connected type}, i.e., a direct product of a simply connected semisimple group $G_0$ and a torus~$C$.

Let $H_0$ be the projection of $H$ to $G_0$ so that $H\subseteq H_0\times C$.
We consider the image $H_{00}^{\sigma}$ of the projection map $H^{\sigma}\to H_0^{\sigma}$: this is an open and closed real Lie subgroup of $H_0^{\sigma}$, which may differ from $H_0^{\sigma}$ only if $H_0$ admits a torus as a quotient group. (Otherwise just $H\supseteq H_0$.)
Let $(H_0^\sigma)^\circ$ denote the identity component of $H_0^\sigma$ for the Hausdorff topology. The following inclusions hold
$$
(H_0^\sigma)^\circ\subseteq H_{00}^\sigma\subseteq H_0^\sigma.
$$

We have two chains of equivariant maps
\begin{equation}\tag{A}
G^{\sigma}/H^{\sigma}\longrightarrow G_0^{\sigma}/H_{00}^{\sigma}\longrightarrow G_0^{\sigma}/H_0^{\sigma}
\end{equation}
and
\begin{equation}\tag{B}
G_0^{\sigma}/\bigl((H_0^{\sigma})^{\circ}\cdot Z(G_0^{\sigma})\bigr)\longleftarrow G_0^{\sigma}/(H_0^{\sigma})^{\circ}\longrightarrow G_0^{\sigma}/H_{00}^{\sigma},
\end{equation}
where all maps except the first one in (A) are finite coverings.

Although the homogeneous spaces in (B) are not in general the real loci of complex algebraic $G_0$-homogeneous spaces defined over $\RR$, one may consider $B_0^{\sigma}$-orbits on them and define the reflection operators in the same way as it is done for real Borel orbits on real loci of complex algebraic homogeneous spaces (see Section~\ref{sec:min&refl}).

\begin{lemma}\label{lem:red-00}
There are bijections compatible with the reflection operators between the following sets of real Borel orbits:
\begin{enumerate}
\item
the set of $B^{\sigma}$-orbits in $G^{\sigma}/H^{\sigma}$ and the set of $B_0^{\sigma}$-orbits in $G_0^{\sigma}/H_{00}^{\sigma}$;
\item
the sets of $B_0^{\sigma}$-orbits in $G_0^{\sigma}/(H_0^{\sigma})^{\circ}$ and in $G_0^{\sigma}/\bigl((H_0^{\sigma})^{\circ}\cdot Z(G_0^{\sigma})\bigr)$.
\end{enumerate}
\end{lemma}

\begin{proof}
(1) follows from the natural bijection between the orbit sets for the action
of $H^{\sigma}$ on $G^{\sigma}/B^{\sigma}=G_0^{\sigma}/B_0^{\sigma}$
and that of $H_{00}^{\sigma}$ on $G_0^{\sigma}/B_0^{\sigma}$ by noticing that $C$ acts trivially on $G/B$.

For (2), note that $Z(G_0^{\sigma})\subset B_0^{\sigma}$ acts transitively on the fibers of the first covering map in~(B). 
\end{proof}

As for $B_0^{\sigma}$ acting on $G_0^{\sigma}/H_0^{\sigma}$, the orbits of this action are the images of $B_0^{\sigma}$-orbits in $G_0^{\sigma}/H_{00}^{\sigma}$ and the latter orbits are the images of $B_0^{\sigma}$-orbits in $G_0^{\sigma}/(H_0^{\sigma})^{\circ}$.
But these correspondences are not bijective in general: the preimage of a $B_0^{\sigma}$-orbit in $G_0^{\sigma}/H_0^{\sigma}$ is a union of at most $|H_0^{\sigma}/H_{00}^{\sigma}|$ orbits of $B_0^{\sigma}$ in $G_0^{\sigma}/H_{00}^{\sigma}$ and the preimage of any $B_0^{\sigma}$-orbit in $G_0^{\sigma}/H_{00}^{\sigma}$ is a union of at most $|H_{00}^{\sigma}/(H_0^{\sigma})^{\circ}|$ orbits of $B_0^{\sigma}$ in $G_0^{\sigma}/(H_0^{\sigma})^{\circ}$. Still, these correspondences between $B_0^{\sigma}$-orbits agree with the reflection operators.

Thus, using Lemma~\ref{lem:red-00}, we reduce the proof to considering the action of simple reflections on the set of $B_0^{\sigma}$-orbits in $G_0^{\sigma}/\bigl((H_0^{\sigma})^{\circ}\cdot Z(G_0^{\sigma})\bigr)$. Consequently, we may assume $G$ to be semisimple.

We shall prove that the braid relations are verified for the simple reflections as in the theorem. 
By applying Remark~\ref{rmk:rk-red-real} for $S=\{\alpha,\beta\}$, we reduce checking the braid relation for two given simple reflections $s_{\alpha}$ and $s_{\beta}$ to groups $G$ of semisimple rank $2$.
The above reductions allow us to assume that $G$ is semisimple of rank 2 and simply connected, and $H$ is connected. Also, by assumption, both simple roots $\alpha_1,\alpha_2$ of $G$ belong to $\Sigma_X\cup\frac12\Sigma_X$.

We now list all possible cases, after recalling a few more notions and facts; for details and references, one may consult~\cite{tim}.
The complement of $\Omax$ in $X$ consists of finitely many prime divisors, each of which being $B$-stable.
The normalizer $N_G(H)$ of $H$ in $G$ acts naturally on the set of these components.
\emph{The spherical closure of $H$}, denoted by $\overline{H}$, is the kernel of this action; $\overline{H}$ obviously contains $H$ and $G/\overline{H}$ is a spherical homogeneous space. Under our assumptions, $G/H$ and $G/\overline{H}$ have one and the same set of spherical roots.

By a result of Knop (\cite[Cor.\,7.6]{knop:auto}), $G/\overline{H}$ is a \emph{wonderful} $G$-homogeneous space.

\begin{lemma}
The list of wonderful $G$-homogeneous spaces with connected stabilizers whose spherical roots are proportional to $\alpha_1$ and $\alpha_2$
consists of the cases A6 (first), C4 (first), C7, C8 (first), C9, G1, G4 (both), G5 of the tables in~\cite{was}.

Moreover, the spherical non-wonderful $G$-homogeneous spaces $G/H$ with the same properties
are such that $H=\overline{H}^\circ$ with $\overline{H}=H\cdot Z(G)$ of type A1, A7 (first) or A8.
\end{lemma}

\begin{proof}
Let us recall that for any wonderful $G$-variety $Y$, the weight lattice $\Xi(Y)$ is generated by the set of spherical roots $\Sigma_Y$ of $Y$.
It follows that $\rk Y$ equals the cardinality of the set $\Sigma_Y$, and in turn, $\rk X=2$ if $X$ is wonderful.
A list of subgroups $K$ of $G$ such that $G/K$ is wonderful is thus contained in Wasserman's classification of wonderful varieties of rank 2  in \cite{was}. The first assertion thus follows from a careful check of the lists given therein.

As for the second assertion, recall that $H\subseteq H\cdot Z(G)\subseteq\overline{H}$ and that
$\Sigma_{G/H}=\Sigma_{G/\overline{H}}$ spans a sublattice of finite index in $\Xi(G/H)$, which implies finiteness of $\overline{H}/H$. From this and Wasserman's lists, we get $\overline{H}=H\cdot Z(G)$. The lemma follows.
\end{proof}

Since the weight lattice of $X$ has rank $2$,  by Proposition~\ref{prop:paramet-borel-orb},
there are either one, two or four open $B^{\sigma}$-orbits in $X^{\mu}$.
Specifically, there is/are one, two or four open $B^{\sigma}$-orbits in $X^{\mu}$
whenever $G/\overline{H}$ is of type A6, A7, G4 and G5; A8, C7, C8 (1st) and C9; or A1, C4 (1st) and G1, respectively.

Let us first check the braid relation for the reflection operators acting on open $B^{\sigma}$-orbits in $X^{\mu}$. If there is a single open $B^{\sigma}$-orbit, then the braid relation is trivially satisfied.
For $G$ of type $\AAA_1\times\AAA_1$, $\CCC_2$ or $\GGG_2$, if there are two open $B^{\sigma}$-orbits, then the simple reflections are acting as two commuting involutions. Taking their product to an even power thus yields the identity hence the braid relation holds in this case.
We are thus left to consider the cases where there are four orbits.
But these cases belong to the class of symmetric spaces, where the braid relations are satisfied by Theorem~\ref{thm:real-Knop-action} and Corollary~\ref{cor:thm-real} (which are proved independently of the results of \S\ref{sec:braid-rel}).

It remains to examine the braid relation for the reflection operators acting on open $B^{\sigma}$-orbits in $G^{\sigma}/H_{00}^{\sigma}$, where $H_{00}^{\sigma}=(H^{\sigma})^{\circ}\cdot Z(G^{\sigma})$, in case $H$ admits a torus as a quotient group.
We assume that $H_{00}^{\sigma}$ is not the real locus of a spherical subgroup of $G$ (otherwise, we fall in the previous situation) and that $H$ is the Zariski closure of $H_{00}^{\sigma}$.
By checking the lists in~\cite{was}, we get that $G/\overline{H}$ is of type A6, A7, C4 (1st), C7, C8 (1st), G4 (both) or G5.

The cases A6 and A7 are excluded by assumption.

For the cases C7 and C8 (1st), we have two open $B^{\sigma}$-orbits in $G^{\sigma}/H^{\sigma}$ and each of them is covered by at most (in fact, exactly) two open $B^{\sigma}$-orbits in $G^{\sigma}/H_{00}^{\sigma}$. The reflection operators may act on the set of open $B^{\sigma}$-orbits either as two commuting involutions or possibly as a transposition of two orbits and the product of two other commuting transpositions. (In fact, the second situation never occurs.) The braid relation of type $\CCC_2$ is satisfied in both cases.

In the case G5, there is a single open $B^{\sigma}$-orbit in $G^{\sigma}/H^{\sigma}$ covered by at most (in fact, exactly) two open $B^{\sigma}$-orbits in $G^{\sigma}/H_{00}^{\sigma}$. The reflection operators act on the set of open $B^{\sigma}$-orbits as two commuting involutions so that the braid relation of type $\GGG_2$ holds.

It remains to work out the first C4 case as well as the cases G4. Note that for any of these cases, we have four open $B^{\sigma}$-orbits in $X^{\mu}$ .

\subsubsection*{Case C4 (1st)}
Here $G=Sp_4$, $H=GL_2$ and $X=G/H$ is a symmetric space.
We consider, as model for $X$, the space of pairs $(V,V')$ of complementary Lagrangian subspaces in $\CC^4$. We assume that the symplectic form $\omega$ takes the following values on the standard basis $(e_1,e_2,e_{-2},e_{-1})$ of $\CC^4$: $\omega(e_{\pm i},e_{\mp i})=\pm1$ whenever $i>0$ and $\omega(e_j,e_k)=0$ otherwise.

Recall that the $G$-equivariant real structures on $X$ are parameterized by the first Galois cohomology set of $\RR$ with values in the equivariant automorphism group $N_G(H)/H$ of $X$ \cite[III.1.3]{Galois}. Since the latter group has order $2$, there are two real structures on $X$: $\mu_1(V,V')=(\overline{V},\overline{V'})$ and  $\mu_2(V,V')=(\overline{V'},\overline{V})$, where the bar indicates complex conjugation in~$\CC^4$.

The real locus $X^{\mu_2}$ consists of all pairs $(V,\overline{V})$
with $V$ being a Lagrangian subspace of $\CC^4$ containing no real vectors.
Then $\kappa(v_1,v_2)=\omega(v_1,\overline{v_2})$ defines a non-degenerate skew-Hermitian form on $V$.
The four open $B^{\sigma}$-orbits in $X^{\mu_2}$ are distributed between three $G^{\sigma}$-orbits which are distinguished by the inertia of $\kappa$.
Here $H^{\sigma}$ equals the unitary group $U_2$ or $U_{1,1}$ hence it is connected and therefore coincides with $H_{00}^{\sigma}$.

The real locus $X^{\mu_1}$ is identified with the space of pairs of complementary Lagrangian subspaces of $\RR^4$ on which $G^{\sigma}=Sp_4(\RR)$ acts transitively.
Here $H^{\sigma}=GL_2(\RR)$ and $H_{00}^{\sigma}=GL_2(\RR)^{\circ}$ is the group of $2\times2$ matrices with positive determinant.

A model of $G^{\sigma}/H_{00}^{\sigma}$ is the space of all quadruples $(V,\RR^+\delta,V',\RR^+\delta')$ where $\delta$ and $\delta'$ are nonzero 2-forms on $V$  and $V'$ respectively and adjusted so that $\delta\wedge\delta'=\omega\wedge\omega$ holds. In other words, we consider pairs of compatibly oriented complementary Lagrangian subspaces in~$\RR^4$.

For a Borel subgroup in $G$, we take the group $B$ of upper triangular matrices in $Sp_4$. Let $T$ be the maximal torus in $B$ consisting of diagonal matrices. If $\pm\eps_i$ are the eigenweights of $e_{\pm i}$ with respect to $T$, then the simple roots of $G$ are $\alpha_1=\eps_1-\eps_2$ and $\alpha_2=2\eps_2$, and the fundamental weights are $\omega_1=\eps_1$ and $\omega_2=\eps_1+\eps_2$.

The open $B$-orbit $\Omax$ of $X$ is defined by the inequalities
$$
x_{-2,-1}\ne0,\qquad x'_{-2,-1}\ne0,\qquad\text{and}\qquad
D=\begin{vmatrix}
x_{2,-1}  & x'_{2,-1}  \\
x_{-2,-1} & x'_{-2,-1} \\
\end{vmatrix}\ne0,
$$
where $x_{jk}$ and $x'_{jk}$ are the Pl\"ucker coordinates of $V$ and $V'$, respectively. 

The boundary divisor of $X$ in the product of two Lagrangian Grassmannians is given by the equation $\Delta=0$ with
$$
\Delta=x_{1,2}x'_{-2,-1}-x_{1,-2}x'_{2,-1}+x_{1,-1}x'_{2,-2}+
x_{2,-2}x'_{1,-1}-x_{2,-1}x'_{1,-2}+x_{-2,-1}x'_{1,2}.
$$
The functions $f_{2\omega_1}=D/\Delta$ and $f_{2\omega_2}=x_{-2,-1}x'_{-2,-1}/\Delta$ of respective $B$-eigenweights $2\omega_1$ and $2\omega_2$ generate the multiplicative group $\CC(X)^{(B)}/\CC^{\times}$. In particular, we have $x_{a,b}=(V,V')\in\Omax$ with
$$
V=\langle e_{-1}+ae_1,e_{-2}+be_2\rangle,\qquad V'=\langle e_{-1}-ae_1,e_{-2}-be_2\rangle, \qquad a,b\ne0,
$$
and $f_{2\omega_1}(x_{a,b})=1/2a$,  $f_{2\omega_2}(x_{a,b})=-1/4ab$.

The four $B^{\sigma}$-orbits in $\Omax^{\mu_1}$ are represented by $x_{a,b}$ as above with $a,b\in\RR$ and are distinguished by the signs of $f_{2\omega_1}$ and $f_{2\omega_2}$ or, equivalently, by the signs of $a$ and $b$.
Note that the matrix $g=\diag(-1,1,1,-1)\in B^{\sigma}$ preserves $x_{a,b}$ but changes the orientations of both $V$ and $V'$.
Each open $B^{\sigma}$-orbit in $X^{\mu_1}$ is thus covered by a single open $B^{\sigma}$-orbit in $G^{\sigma}/H_{00}^{\sigma}$ hence the braid relation holds for $G^{\sigma}/H_{00}^{\sigma}$ since it does for $G^{\sigma}/H^{\sigma}$.

\subsubsection*{Cases G4}
In both cases, there is a single open $B^{\sigma}$-orbit in $G^{\sigma}/H^{\sigma}$
and four open ones in $G^{\sigma}/H_{00}^{\sigma}$. We only work out the second G4 case; the first case is a particular instance of an example considered in \S\ref{ex:G/TU'}.

Here $G$ is a simple algebraic group of type $\GGG_2$ and
$H=T\Ru{H}$ with
$\Ru{H}=\prod U_{-\alpha}$ over all positive roots $\alpha\ne\alpha_1,\alpha_1+\alpha_2$. (Here $\alpha_1$ and $\alpha_2$ denote the short and long simple root, respectively.)
The open $B$-orbit of $X=G/H$ is
$\Omax=U U_{-\alpha_1}^{\times} U_{-\alpha_1-\alpha_2}^{\times}x_0$
with $x_0\in X$ being the base point of $X$ 
and $U_{-\alpha}^{\times}=U_{-\alpha}\setminus\{1\}$. 
For the slice, we take
$Z=U_{-\alpha_1}^{\times}U_{-\alpha_1-\alpha_2}^{\times}x_0$.
The action of $T$ on $Z$ is described as follows: given $t\in T$ 
and $x\in Z$ with $x=u_{-\alpha_1}(a_1)u_{-\alpha_1-\alpha_2}(a_2)x_0$
($a_i\in\CC^\times$), 
we get
$$
tx=u_{-\alpha_1}\bigl(a_1/\alpha_1(t)\bigr)\, u_{-\alpha_1-\alpha_2}\bigl(a_2/\alpha_1(t)\alpha_2(t)\bigr)\,x_0.
$$

Since $H$ is self-normalizing in $G$, there is a unique real structure on $X$, namely the one for which $x_0$ is a real point. We have $H^{\sigma}=\0T2\cdot H^{\sigma}_{00}$. Let $\tilde{x}_0$ denote the base point of $\widetilde X^{\mu}=G^{\sigma}/H_{00}^{\sigma}$. Then the preimage of $\Omax^{\mu}$ in $\widetilde X^{\mu}$ is $\OOmax^{\mu}=
U^{\sigma}\cdot(U_{-\alpha_1}^{\times})^{\sigma}\cdot(U_{-\alpha_1-\alpha_2}^{\times})^{\sigma}\cdot\0T2\cdot\tilde{x}_0$.

The $B^{\sigma}$-orbits in $\OOmax^{\mu}$ correspond bijectively to the $T^{\sigma}$-orbits in $$\widetilde Z^{\mu}= (U_{-\alpha_1}^{\times})^{\sigma}\cdot(U_{-\alpha_1-\alpha_2}^{\times})^{\sigma}\cdot\0T2\cdot\tilde{x}_0,$$ cf.\ Proposition~\ref{prop:paramet-borel-orb}.
The action of $T^{\sigma}$ on $\widetilde Z^{\mu}$ is given by the formula
$$
x=u_{-\alpha_1}(a_1)\,u_{-\alpha_1-\alpha_2}(a_2)\,t_0\tilde{x}_0\mapsto
tx=u_{-\alpha_1}\bigl(a_1/\alpha_1(t)\bigr)\, u_{-\alpha_1-\alpha_2}\bigl(a_2/\alpha_1(t)\alpha_2(t)\bigr)\,t_0t|t|^{-1}\tilde{x}_0,
$$
where $a_i\in\RR^{\times}$, $t_0\in\0T2$, and $|t|$ is the unique square root of $t^2$ in $(T^{\sigma})^{\circ}$.
This action thus preserves the signs of 
$a_1\cdot\alpha_1(t_0)$ and $a_2\cdot\alpha_1(t_0)\cdot\alpha_2(t_0)$ and, moreover, these signs form a complete system of invariants separating open $B^{\sigma}$-orbits in $\widetilde X^{\mu}$. 

To understand how any simple reflection $s_{\alpha_i}$ acts on each $B^\sigma$-orbit, it suffices to look at open $B^{\sigma}$-orbits meeting a curve $u_{-\alpha_i}(s)\,x$ ($s\in\RR$) with $x\in\widetilde Z^{\mu}$.
A straightforward computation for $i=2$ shows that
\begin{align*}
u_{-\alpha_1}(s)\,x &= u_{-\alpha_1}(a_1+s)\,u_{-\alpha_1-\alpha_2}(a_2)\,t_0\tilde{x}_0, \\
u_{-\alpha_2}(s)\,x &= u_{-\alpha_1}(a_1)\,u_{-\alpha_1-\alpha_2}(a_2-sa_1)\,t_0\tilde{x}_0.
\end{align*}

Thus the action of $U_{-\alpha_1}^\sigma$ may alter the sign of $a_1\cdot\alpha_1(t_0)$ and preserves the sign of $a_2\cdot\alpha_1(t_0)\cdot\alpha_2(t_0)$, and vice versa for the action of $U_{-\alpha_2}^\sigma$.

It follows that the action of $s_{\alpha_i}$ on the set of open $B^{\sigma}$-orbits changes the $i$-th sign in the $2$-tuple of signs mentioned above. Thus $s_{\alpha_1}$ and $s_{\alpha_2}$ act on the set of $2$-tuples of signs as generators of an elementary Abelian 2-group of order $4$
which is compatible with the braid relation of type~$\GGG_2$.

The proof of Theorem~\ref{thm:main-little-Weyl} is now complete.
\end{proof}

\begin{remark}\label{rmk:no-braid}
The braid relation fails indeed in case A6, as shown in \S\ref{ex:G/TU'}, but not in case A7, as we show right below. Observe that in case A6 the simple roots are of type (T2) with respect to any open $B^{\sigma}$-orbit.

In case A7 there is a single open $B^{\sigma}$-orbit in $G^{\sigma}/H^{\sigma}$ and two open $B^{\sigma}$-orbits in $G^{\sigma}/H_{00}^{\sigma}$. The diagrammatic automorphism of $G=SL_3$ preserving $B$ and $H$ and commuting with $\sigma$ induces an automorphism of $G^{\sigma}/H_{00}^{\sigma}$ twisting the action of $G^{\sigma}$ and interchanging the two simple roots. It follows that the respective reflection operators act on the two-point set of open $B^{\sigma}$-orbits similarly and satisfy the braid relation of type $\AAA_2$.
\end{remark}

\section{Weyl group action: the complex case}
\label{sec:Weyl-C}

Reflection operators defined in \S\ref{subsec:min&refl-C} induce an action of the Weyl group $W$ on a certain subset of the set $\B(X)$ of Borel orbits of $X$.
This is a particular case of a result
proved by Knop in~\cite[\S6]{knop:action} by relating
certain $B$-stable subsets of $X$ to the cotangent bundle of $X$; we recall Knop's results below focusing on our setting, that is on a spherical homogeneous space $X$.

\subsection{}\label{subsec:rank}

The subset of $\B(X)$ alluded above is the set $\Bmax(X)$ of $B$-orbits \emph{of maximal rank}.

The notion of rank recalled in~\S\ref{subsec:recalls} is also valid for any irreducible $B$-subvariety $Y$ of $X$.
Namely, the \emph{rank} $\rk{Y}$ of $Y$ is the rank of the lattice given by the weights of the $B$-eigen\-func\-tions in the function field of $Y$. Notice that $Y$ contains an open dense $B$-orbit $\Omax^Y$ and
$$
\rk\Omax^Y=\rk{Y}\leq\rk{X}
$$ as proved in~\cite[2.4]{knop:action}.
Moreover, we have the following characterization of the rank.

\begin{lemma}[{\cite[2.1]{knop:action}}]\label{lem:rank}
The equality $\rk Y =\dim Y-\dim Ux$ holds for any $x\in\Omax^Y$.
\end{lemma}

\subsection{}\label{subsec:rest-cone}

We may reduce our considerations to the case where $X$ is quasi-affine.
Indeed, embed the $G$-variety $X$ equivariantly into $\PP(V)$ for some $G$-module $V$ and consider the punctured cone $\widehat{X}\subset V$ over~$X$.
The variety $\widehat{X}$ is a spherical homogeneous space for the group $G\times\CC^{\times}$.
Moreover, there is a natural bijection between the Borel orbits in $X$ and $\widehat{X}$ preserving the types of the $P_{\alpha}$-spans in the sense of \S\ref{subsec:min&refl-C} and compatible with the reflection operators.

In the sequel, we assume $X$ to be quasi-affine.

\subsection{}

The cotangent bundle $T^*X$ of $X$ comes equipped with the canonical symplectic structure and the induced $G$-action on $T^*X$ is Hamiltonian. In particular, there exists a momentum map $\Phi:T^*X\to\g^*$, which is defined by the formula
$$
\langle\Phi(\psi),\xi\rangle=\langle\psi,\xi_{x}\rangle,\qquad\forall x\in X,\ \psi\in T^*_xX,\ \xi\in\g,
$$
where $\xi_{*}$ stands for the vector field on $X$ naturally induced by $\xi$.

We identify $\g^*$ with $\g$ via a $G$-invariant inner product on $\g$.
Then $\langle\cdot,\cdot\rangle$ in the left hand side of the above formula may be regarded as the inner product, rather than the pairing between $\g^*$ and $\g$.

The momentum map $\Phi$ maps $T_{x_0}^*X$ isomorphically onto the annihilator $\h^{\perp}\subset\g^*$ of $\h$ and
$$
\Phi:T^*X\simeq\bundle{G}{H}{\h^{\perp}}\longrightarrow G\h^{\perp}\subseteq\g^*
$$
is nothing else than the action map.
Namely, let $g*\xi\in\bundle{G}{H}{\h^{\perp}}$ denote the class of $(g,\xi)\in G\times\h^{\perp}$, then
$\Phi(g*\xi)=g\xi$ (the (co)adjoint action).

There is a natural closed immersion
$$
T^*X\hookrightarrow\bundle{G}{H}{\g^*}\simeq X\times\g^*
$$
mapping $\psi\in T_x^*X$ to $(x,\Phi(\psi))$.

\subsection{}

Set
$$
\U=\Phi^{-1}(\uu^{\perp}).
$$
Clearly, $\U\subset T^*X$ is the disjoint union of the conormal bundles of all $U$-orbits in $X$.
The set of $U$-orbits in $X$ subdivides into finitely many foliations, each of which consisting of the $U$-orbits in a given $B$-orbit $\OO\subset X$.
Therefore $\U$ is a disjoint union of finitely many locally closed subvarieties $\U_{\OO}\subset T^*X$, where $\U_{\OO}$ denotes the conormal bundle of the foliation of $U$-orbits in $\OO$.

Note that $\dim\U_{\OO}=\dim X+\rk\OO$, because each conormal bundle has dimension $\dim X$ and $\rk\OO$  equals the dimension of the foliation of $U$-orbits in $\OO$ by Lemma~\ref{lem:rank}.

\subsection{}

Recall the definition of the torus $A$ set up in \S\ref{sec:open-Borel-orb}.
We may identify $\aaa^*\subseteq\ttt^*$ with a subspace of $\g$ or $\g^*$ indifferently.
In what follows, we shall consider an open subset $\aaa^{\pr}\subset\aaa^*$ consisting of sufficiently general points defined by certain open conditions, to be gradually specified later on.
We call $\aaa^{\pr}$ the \emph{principal stratum} of $\aaa^*$. At the first stage, let $\aaa^{\pr}$ denote the open subset of $\aaa^*$ obtained by removing all proper intersections of $\aaa^*$ with the kernels of coroots and with $\9{\aaa^*}w$ ($w\in W$).
Here and below, we denote the natural action of an element of the Weyl group by putting the respective superscript on the left.
\begin{remark}
\label{rmk:pr}
The open subset $\aaa^{\pr}$ thus obtained satisfies the following property:
$$g\zeta=\zeta'\implies g\in N_G(\aaa),\qquad \forall\zeta,\zeta'\in\aaa^{\pr},\ g\in G.$$
Moreover, such a $g$ is unique (for given $\zeta$ and $\zeta'$) up to translation by $Z_G(\aaa)$.
\end{remark}

\begin{proposition}[{\cite[3.2, 3.3]{inv.mot}}]
\label{prop:moment}
\setcounter{savecounter}{\value{equation}}
\setcounter{equation}{0}
\begin{align}
T^*X&=G\U=\overline{G\U_{\Omax}}.\label{eqn:GU}\\
\overline{\Img\Phi}&=G(\Ru\p+\aaa)=\overline{G\aaa}.\label{eqn:cl(Im(mom))}\\
\Img\Phi&\supset G\aaa^{\pr}.\label{eqn:Im(mom)}
\end{align}
\setcounter{equation}{\value{savecounter}}
\end{proposition}

\begin{proof}
It follows from the Local Structure Theorem that the orbits of $U$, $\Ru{P}$ and $\Ru{P}\rtimes L_0$ in $\Omax$ coincide.
Therefore $\Phi(\U_{\Omax})\subseteq(\Ru\p+\llll_0)^{\perp}=\Ru\p+\aaa$.
Besides, $T_{x_0}\Omax=T_{x_0}\Ru{P}x_0\oplus T_{x_0}Lx_0\simeq\Ru\p\oplus\aaa$ whence the composite map
$$\xymatrix{
\U_{\Omax}\ar[r]^-{\Phi} & \Ru\p+\aaa \ar[rr]^-{\text{projection}} && \aaa
}$$
is surjective.

Consider the solvable Lie algebra $\Ru\p+\aaa$.
Given any $\xi\in\Ru\p+\aaa$, its semisimple part is $\Ru{P}$-conjugate to the projection of $\xi$ to~$\aaa$.
But $\z_{\g}(\zeta)=\z_{\g}(\aaa)=\llll$ for any $\zeta\in\aaa^{\pr}$ (here the quasi-affinity is used, see e.g.~\cite[3.1]{inv.mot})
whence $\zeta$ cannot be the semisimple part of some $\xi\in\Ru\p+\aaa$ with a nonzero nilpotent part.
It follows that $\Ru\p+\aaa^{\pr}=\Ru{P}\aaa^{\pr}=P\aaa^{\pr}$ consists of semisimple elements and is contained in $\Phi(\U_{\Omax})$, by surjectivity of $\U_{\Omax}\to\aaa$ and $P$-equivariance of $\Phi$.
By $G$-equivariance of $\Phi$, we get~\eqref{eqn:Im(mom)}. Also we get $\overline{\Phi(\U_{\Omax})}=\Ru\p+\aaa$.

Since $[\Ru{\p^-},\zeta]=\Ru{\p^-}$, $\forall\zeta\in\aaa^{\pr}$, $\Ru{P^-}$ moves general points of $\Ru\p+\aaa$ in a direction transversal to $\Ru\p+\aaa$.
Hence $\Ru{P^-}$ moves general points of $\U_{\Omax}$ in a direction transversal to $\U_{\Omax}$. But $\codim\U_{\Omax}=\dim X-\rk X=\dim\Ru{P}$ whence $\Ru{P^-}\U_{\Omax}$ is dense in $T^*X$.
Consequently, $G(\Ru\p+\aaa^{\pr})=G\aaa^{\pr}$ is dense in $\Img\Phi$. We conclude the proof of \eqref{eqn:GU} and \eqref{eqn:cl(Im(mom))} by noticing that $\U$ and $\Ru\p+\aaa$ are $B$-stable closed subsets of $T^*X$ and $\g$, respectively,
whence their $G$-spans are closed, too.
\end{proof}

\subsection{}\label{subsec:polarization}

Consider the following Cartesian square:
$$
\xymatrix{
\widetilde{T}^*X :=T^*X\mathbin{\underset{\smash{\ttt^*/W}}\times}\ttt^*\ar[r] \ar[dd]  & T^*X \ar[d]_{\Phi} \\
& \g^* \ar[d] \\
**[l] \ttt^* \ar[r] & **{!/r4ex/} \ttt^*/W\simeq\g^*\quot{G}.
}
$$

The horizontal arrows are the quotient maps by the actions of $W$, the action on the top being naturally induced from the bottom by the base change.

\begin{lemma}[{\cite[6.5]{knop:action}}]
\label{lem:knop}
The irreducible components of $\widetilde{T}^*X$ map onto $T^*X$ and are transitively permuted by $W$.
\end{lemma}

The subvariety $\U$ lifts isomorphically to a subvariety $\widetilde\U$ of $\widetilde{T}^*X$ by the mapping $\psi\mapsto\widetilde\psi=(\psi,\zeta)$, where
$\Phi(\psi)=\eta+\zeta\in\uu^{\perp}=\bb=\uu\oplus\ttt$.
By $\widetilde{\U}_{\OO}$, we denote  the image of $\U_{\OO}$ in $\widetilde{\U}$.
It follows from Proposition~\ref{prop:moment}, \eqref{eqn:GU} that $\widetilde{\U}_{\Omax}$ is contained in a unique irreducible component of $\widetilde{T}^*X$; we denote this component by $\widehat{T}^*X$  and call it the \emph{polarized cotangent bundle}.

Thanks to Proposition~\ref{prop:moment}, \eqref{eqn:cl(Im(mom))} and \eqref{eqn:Im(mom)}, and Remark~\ref{rmk:pr}, the image of $T^*X$ in $\ttt^*/W$ contains the dense open subset $\aaa^{\pr}/W(\aaa)$, where
$$
W(\aaa):=N_G(\aaa)/Z_G(\aaa)=N_W(\aaa)/Z_W(\aaa)
$$
is a finite group acting freely on $\aaa^{\pr}$.
We call the preimages of $\aaa^{\pr}/W(\aaa)$ through the various maps under consideration \emph{principal strata} and denote them by adding the superscript ${}^\pr$ to the underlying domains.
In particular, we have
$$
\widetilde{T}^{\pr}X=T^{\pr}X\mathbin{\underset{\aaa^{\pr}/W(\aaa)}\times}\bigsqcup_w\9{\aaa^{\pr}}w,
$$
where the union is taken over representatives of left cosets in $W/N_W(\aaa)$.

The map $\widetilde{T}^{\pr}X\to T^{\pr}X$ is an \'etale finite covering. The variety $\widetilde{T}^{\pr}X$ is smooth, open and dense in $\widetilde{T}^*X$ but disconnected in general.
The connected component $\widehat{T}^{\pr}X$ of $\widetilde{T}^{\pr}X$ containing $\widetilde\U_{\Omax}^{\pr}$ is open in $\widehat{T}^*X$.
The map $\widehat{T}^{\pr}X\to T^{\pr}X$ is an \'etale Galois covering with Galois group the little Weyl group of $X$  
\cite[\S3]{inv.mot}.

The stabilizer of the polarized cotangent bundle $\widehat{T}^*X$ (or $\widehat{T}^{\pr}X$) in $W$ is $W_X\ltimes W_L$, where $W_L=Z_W(\aaa^{\pr})$ is the Weyl group of $L$; $W_L$ fixes $\widehat{T}^*X$ pointwise whereas $W_X$ acts on $\widehat{T}^*X$ effectively.

\begin{proposition}\label{prop:component}
The set $\Phi^{-1}(\aaa^{\pr})$ is a smooth closed subvariety of $T^{\pr}X$ and its connected components are transitively permuted by $N_G(\aaa)$. There is a unique connected component $\SSS$ of $\Phi^{-1}(\aaa^{\pr})$ intersecting $\U_{\Omax}$. Put $\widetilde\SSS=\SSS\times_{\aaa^{\pr}}\aaa^{\pr}\subset\widetilde{T}^{\pr}X$.
Then
$$
T^{\pr}X\simeq\bundle{G}{N_X}{\SSS}\quad\text{and}\quad
\widehat{T}^{\pr}X\simeq\bundle{G}{L}{\widetilde\SSS},
$$
where $N_X\subseteq N_G(\aaa)$ is the preimage of $W_X\subseteq W(\aaa)$.
\end{proposition}

\begin{proof}
Since $G\aaa^{\pr}\simeq\bundle{G}{N_G(\aaa)}\aaa^{\pr}$ by Remark~\ref{rmk:pr}, we have $T^{\pr}X\simeq\bundle{G}{N_G(\aaa)}{\Phi^{-1}(\aaa^{\pr})}$. This yields the claim about $\Phi^{-1}(\aaa^{\pr})$ and its connected components.

The variety $\widetilde{T}^{\pr}X$~contains the
open and closed subset
$$
T^{\pr}X\mathbin{\underset{\aaa^{\pr}/W(\aaa)}\times}\aaa^{\pr}\simeq \bundle{G}{L}{\Bigl(\Phi^{-1}(\aaa^{\pr})\mathbin{\underset{\aaa^{\pr}}\times}\aaa^{\pr}\Bigr)}.
$$
This isomorphism can be proved as follows.
First, note that $T^{\pr}X\times_{\aaa^{\pr}/W(\aaa)}\aaa^{\pr}$ consists of points of the form $(g*\psi,\zeta)$ with $\Phi(\psi)=\zeta\in\aaa^{\pr}$.
Besides, we have
\begin{multline*}
(g'*\psi',\zeta')=(g*\psi,\zeta) \iff \zeta'=\zeta,\ \psi'=n\psi,\ g'=gn^{-1}\ (n\in N_G(\aaa)) \\
\implies \zeta=\zeta'=n\zeta \implies n\in Z_G(\aaa)=L.
\end{multline*}

Since $\widehat{T}^{\pr}X$ is a connected component of $T^{\pr}X\times_{\aaa^{\pr}/W(\aaa)}\aaa^{\pr}$ and $L$ is connected, it follows that $\widehat{T}^{\pr}X\simeq\bundle{G}{L}{\widetilde\SSS}$ for the unique connected component  $\widetilde\SSS$ of $\Phi^{-1}(\aaa^{\pr})\times_{\aaa^{\pr}}\aaa^{\pr}$ intersecting $\widetilde\U_{\Omax}$. Specifically, $\widetilde\SSS=\SSS\times_{\aaa^{\pr}}\aaa^{\pr}$, where $\SSS$ is the unique connected component of $\Phi^{-1}(\aaa^{\pr})$ intersecting $\U_{\Omax}$.

To prove $T^{\pr}X\simeq\bundle{G}{N_X}{\SSS}$, take
$\psi,\psi'\in\SSS$ and $\psi'=n\psi$ with $n\in N_G(\aaa)$.
Set $\zeta=\Phi(\psi)$ and $\zeta'=\Phi(\psi')$.
Then $(\psi,\zeta),(\psi',\zeta')\in\widetilde\SSS$ and $\zeta'=\9{\zeta}w$, with $w$ being the image of $n$ in $W(\aaa)$.
But then $(\psi,\zeta')\in\widehat{T}^{\pr}X$ whence $w\in W_X$ and $n\in N_X$.
\end{proof}

\begin{proposition}
\label{prop:flats}
For general $\zeta\in\aaa^{\pr}$, the set $\SSS(\zeta)=\SSS\cap{}\Phi^{-1}(\zeta)$ consists of a single $L$-orbit.
It projects isomorphically onto a closed $L$-orbit $Z(\zeta)\subset X$ contained in $\Omax$.
The variety $Z(\zeta)$, called a \emph{flat}, may be taken for a slice in the Local Structure Theorem.
\end{proposition}

\begin{proof}
From $\Phi(\U_{\Omax}^{\pr})=\Ru\p+\aaa^{\pr}=P\aaa^{\pr}\simeq\bundle{P}{L}{\aaa^{\pr}}$, we get $$\U_{\Omax}^{\pr}\simeq\bundle{P}{L}{(\SSS\cap T^*\Omax)}.$$
Besides, the Local Structure Theorem implies $$\U_{\Omax}\simeq\bundle{P}{L}{\U|_Z},$$ where $\U|_Z$ is the restriction of $\U_{\Omax}$ to $Z$.
Hence $\SSS\cap T^*\Omax$ and $\U^{\pr}|_Z$ are two cross-sections for the free $\Ru{P}$-action on $\U_{\Omax}^{\pr}$. We have a commutative $L$-equivariant diagram
$$\xymatrix{
\U|_Z \ar[d]_{\Phi} \ar[rr]^-{\sim} && **[l] T^*Z\simeq Z\times\aaa^* \ar@<-1.2ex>[d]^{\text{projection}} \\
\Ru\p+\aaa \ar[rr]^-{\text{projection}} && **[l] \aaa\simeq\aaa^*.
}$$
Taking the quotient by $\Ru{P}$ of $\U_{\Omax}^{\pr}$, we see that $\SSS\cap T^*\Omax\simeq Z\times\aaa^{\pr}$ (as $L$-varieties)
and for any $\zeta\in\aaa^{\pr}$ the subvariety $\SSS(\zeta)\cap T^*\Omax\simeq Z$ projects isomorphically onto an $L$-orbit $Z(\zeta)\subset\Omax$ intersecting each $\Ru{P}$-orbit in a single point.

For general $\zeta$ the flat $Z(\zeta)$ is closed in $X$ \cite[7.8]{inv.mot}
whence $\SSS(\zeta)\subset T^*\Omax$ is a single $L$-orbit.
Under the closed embedding $T^*X\hookrightarrow X\times\g^*$, the closed subvariety $\SSS(\zeta)$ identifies with $Z(\zeta)\times\{\zeta\}$. This concludes the proof.
\end{proof}

From now on, we specify the definition of $\aaa^{\pr}$ by passing to an open subset whose points satisfy Proposition~\ref{prop:flats}. All previous statements about principal strata remain valid.

\subsection{}

Let $\widetilde W$ be the group generated by the simple reflections
$s_{\alpha}$ of $W$ with relations $s_{\alpha}^2=1$.
Then $W$ is the quotient group of $\widetilde{W}$ defined by the braid relations between the generators.

The following theorem gathers Theorem 6.2 and Theorem 6.4 in~\cite{knop:action}.

\begin{theorem}
\label{thm:B-orb&comp-C}
The assignment  $\OO\mapsto \overline{G\widetilde\U_{\OO}}$
defines a bijection between $\Bmax(X)$ and the set of irreducible components of $\widetilde{T}^{*}X$. Moreover, this map is $\widetilde{W}$-equivariant.

In particular, the reflection operators induce a transitive $W$-action on the set $\Bmax(X)$
and the stabilizer of $\Omax$ under this action is $W_X\ltimes W_L$.

\end{theorem}

\begin{corollary}\label{cor:conn-comp}
The map $\OO\mapsto G\widetilde\U_{\OO}^{\pr}$ is a bijection between $\Bmax(X)$ and the set of connected components of $\widetilde{T}^{\pr}X$.
Moreover this map is $\widetilde{W}$-equivariant.
\end{corollary}

\begin{proof}
By Proposition~\ref{prop:component}, any connected component of $\widetilde{T}^{\pr}X$ is of the form $\bundle{G}{L}{\,\9{\widetilde\SSS}w}$ (${w\in W}$).
It consists of points of the form $\widetilde\varphi=(g*\psi,\9{\zeta}w)$ with $\psi\in\SSS$, $\Phi(\psi)=\zeta\in\aaa^{\pr}$. Note that
$$
\widetilde\varphi\in\widetilde\U \iff g\zeta=\eta+\9{\zeta}w\in\uu\oplus\ttt
\iff g\zeta=u(\9{\zeta}w),\ u\in U \iff g=unl,
$$
where $n\in N_G(T)$ represents $w\in W$ and $l\in L$.
Consequently $\widetilde\U\cap\9{\widehat{T}^{\pr}X}w=U\cdot n(\9{\widetilde\SSS}w)$.

We claim that $\widetilde\U_{\OO}^{\pr}=\widetilde\U\cap\9{\widehat{T}^{\pr}X}w$ for $\OO=\9{\Omax}w$. Since the latter intersection is $B$-stable and closed in $\9{\widehat{T}^{\pr}X}w$,
this equality together with Theorem~\ref{thm:B-orb&comp-C} implies $\9{\widehat{T}^{\pr}X}w=G\widetilde\U_{\OO}^{\pr}$, which yields the corollary.

Let us prove the claim. By Theorem~\ref{thm:B-orb&comp-C},
we have $\widetilde\U_{\OO}^{\pr}\subseteq\widetilde\U\cap\9{\widehat{T}^{\pr}X}w$. To prove the reverse inclusion, it suffices to show that $n\SSS\subset\U_{\OO}^{\pr}$.
Observe that $\9{\widehat{T}^*X}w$ projects onto $\9{\aaa^*}w$ under the second projection in the fiber product.
Since $G\widetilde\U_{\OO}$ is dense in $\9{\widehat{T}^{*}X}w$, we have $\Phi(\U_{\OO})\subseteq\uu+\9{\aaa}w$ and the composite map
$$\xymatrix{
\U_{\OO}\ar[r]^-{\Phi} & \uu+\9{\aaa}w \ar[rr]^-{\text{projection}} && \9{\aaa}w
}$$
is dominant. By $B$-equivariance and fiberwise linearity of the momentum map, it is even surjective. Since $\Img\Phi\cap(\uu+\9{\aaa^{\pr}}w)\subset(\Img\Phi)^{\pr}$ consists of semisimple elements,
we have $\Phi(\U_{\OO}^{\pr})=B(\9{\aaa^{\pr}}w)\supset\9{\aaa^{\pr}}w$.
Hence $\U_{\OO}^{\pr}\subseteq Un\SSS$ intersects each $n\SSS(\zeta)\subseteq\Phi^{-1}(\9{\zeta}w)$ (${\zeta\in\aaa^{\pr}}$).
But the latter is a single $T$-orbit by Proposition~\ref{prop:flats} whence $n\SSS(\zeta)\subset\U_{\OO}^{\pr}$.
\end{proof}

\begin{corollary}[of the proof] \label{cor:oftheproof}
For any $x\in Z=Z(\zeta)$ ($\zeta\in\aaa^{\pr}$) and any $n\in N_G(T)$ we have $nx\in\OO=\9{\Omax}w$, where $w$ is the image of $n$ in $W$.
\end{corollary}

\begin{proof}
The claim follows readily from the inclusion $n\SSS(\zeta)\subset\U_{\OO}^{\pr}$.
\end{proof}

In particular, for any $n\in N_X$ and $x\in Z$, we may define another point $x'=n\star x\in Z$ by the formula $nx=ux'$ (for a unique $u\in\Ru{P}$).

\begin{question}\label{que:N_X-action}
Does this define an action of $N_X$ on $Z$?
\end{question}

The answer is affirmative if each $w\in W_X$ can be represented by an element of $H$ (e.g., for symmetric spaces, see \S\ref{ex:symm}): $Z$ is just preserved by $N_X$ and the $\star$-action is the usual action of $N_X$ on $Z$.

\section{Weyl group action: the real case}
\label{sec:R-tangent}

In this section, we extend Knop's action of the Weyl group $W$ (as recalled in the previous section) to the set of $B^{\sigma}$-orbits contained in the real loci of $B$-orbits of maximal rank.
We shall show that this is possible 
for a special class of varieties that we precisely identify.

\subsection{}

We first observe that the punctured cone construction over $X$ from \S\ref{subsec:rest-cone} can be carried over~$\RR$.
There is a bijective correspondence between the orbits of $B^{\sigma}$ (resp.\ $G^{\sigma}$) on $X^{\mu}$ and
the orbits of $B^{\sigma}\times\RR^{\times}$ (resp.\ $G^{\sigma}\times\RR^{\times}$) on $\widehat{X}^{\mu}$ compatible with the types of $P_{\alpha}^{\sigma}$-spans and with the reflection operators, as introduced in \S\ref{subsec:min&refl-R}.
Thus we may again assume $X$ to be quasi-affine.

\subsection{}

The real structure on $X$ lifts to $T^*X$.
The momentum map $\Phi$ and the fiber product $\widetilde{T}^*X$ are defined over $\RR$.
The real locus of $T^*X$ is just the real cotangent bundle $T^*X^{\mu}$ of the real analytic manifold $X^{\mu}$ and the real locus of $\widetilde{T}^*X$ is
$$
\widetilde{T}^*X^{\mu}=T^*X^{\mu}\mathbin{\underset{(\ttt/W)^{\sigma}}\times}\ttt^{\sigma}.
$$
The real locus $\widetilde{T}^*X^{\mu}$ is preserved under the action of $W$ on $\widetilde{T}^*X$. The principal stratum $\widetilde{T}^{\pr}X^{\mu}$ is a smooth totally real analytic submanifold of $\widetilde{T}^{\pr}X$.

Note that the projection $\widetilde{T}^{\pr}X^{\mu}\to T^{\pr}X^{\mu}$ is not necessarily surjective:
its image consists of $\psi\in T^*X^{\mu}$ such that $\Phi(\psi)\in (G(\aaa^{\pr})^{\sigma})^{\sigma}=G^{\sigma}(\aaa^{\pr})^{\sigma}$
(the latter equality stems from Lemma~\ref{lem:G/L(R)} below), while
$\Phi$ maps $T^{\pr}X^{\mu}$ to $(G\aaa^{\pr})^{\sigma}$, which is usually bigger than $G^{\sigma}(\aaa^{\pr})^{\sigma}$.

\begin{lemma}
\label{lem:G/L(R)}
The group $G^{\sigma}$ acts transitively on the real locus of $G/L$.
\end{lemma}

\begin{proof}
The Bruhat decomposition implies that every $g\in G$ is uniquely written in the form $g=uvnl$,
where $u\in U\cap\9{\Ru{P^-}}w$, $v\in\9{\Ru{P}}w$, $l\in L$, $w\in W$ is the shortest element in a left coset $wW_L$, and $n\in N_G(T)^{\sigma}$ is a fixed representative of $w$.
If $gL\in G/L$ is a real point, then $\sigma(g)\in gL$ whence $u,v\in G^{\sigma}$, and $g$ may be replaced with $gl^{-1}=uvn\in G^{\sigma}$.
\end{proof}

{\begin{lemma}
\label{lem:L(R)}
The equality $L^{\sigma}=T^{\sigma}\cdot[L,L]^{\sigma}$ holds.
\end{lemma}}

{\begin{proof}
Take $g\in L^{\sigma}$ and decompose it as $g=th$ with $t\in Z(L)^{\circ}\subseteq T$ and $h\in[L,L]$.
Then $g=\sigma(g)=\sigma(t)\sigma(h)$ implies that $\sigma(t)^{-1}t=\sigma(h)h^{-1}$ is an element of the finite group $Z(L)^{\circ}\cap[L,L]$.
Let $s$ be a square root of this latter element in the split maximal torus $T\cap[L,L]$ of~$[L,L]$: $$\sigma(t)^{-1}t=s^2=\sigma(s)^{-1}s.$$
Replacing $t$ with $ts^{-1}$ and $h$ with $sh$ yields a decomposition of $g$ with factors in $T^{\sigma}$ and $[L,L]^{\sigma}$, respectively.
\end{proof}}

\subsection{}

Let $\OO$ be a $B^{\sigma}$-orbit on $X^{\mu}$ and $\OO_{\CC}$ be the $B$-orbit  on $X$ containing $\OO$.
The conormal bundle $\U_{\OO}$ of the foliation given by the $U^{\sigma}$-orbits in $\OO$ is the restriction to $\OO$ of the real locus $(\U_{\OO_{\CC}})^{\mu}$ of $\U_{\OO_{\CC}}$;
it is open and closed in $(\U_{\OO_{\CC}})^{\mu}$ (in the classical topology).
Recall the right inverse map $\U\to\widetilde\U$ of the projection $\widetilde{T}^*X\rightarrow T^*X$ over $\U$ from \ref{subsec:polarization} and denote the image of $\U_\OO$ through this map by $\widetilde\U_\OO$.

Let $\Bmax(X^\mu)$ be the subset of $\B(X^{\mu})$ consisting of $B^\sigma$-orbits $\OO$ such that $\OO_\CC\in\Bmax(X)$.

\begin{proposition}
\label{prop:B-orb&comp-R}
The sets $G^{\sigma}\widetilde\U_{\OO}^{\pr}\subset\widetilde{T}^{\pr}X^{\mu}$, with $\OO\in\Bmax(X^\mu)$, are open and closed.
They form a partition of  $\widetilde{T}^{\pr}X^{\mu}$:
\begin{equation}\label{eqn:partition}
\widetilde{T}^{\pr}X^{\mu}=\coprod_{\OO\in\Bmax(X^\mu)} G^{\sigma}\widetilde\U_{\OO}^{\pr}.
\tag{P}
\end{equation}
\end{proposition}

\begin{proof}
Let $\OO_{\CC}=\9{\Omax}w$ and $n\in N_G(T)^{\sigma}$ represent $w\in W$.
Recall from Corollary~\ref{cor:conn-comp} that each connected component of $\widetilde{T}^{\pr}X$ is of the form
$$
\9{\widehat{T}^{\pr}X}w = G\widetilde\U_{\OO_{\CC}}^{\pr} \simeq \bundle{G}{L}{\9{\widetilde\SSS}w} \simeq \bundle{G}{\9Lw}{n(\9{\widetilde\SSS}w)},
$$
with
$$
\widetilde\U_{\OO_{\CC}}^{\pr}=U\cdot n(\9{\widetilde\SSS}w).
$$

By Lemma~\ref{lem:G/L(R)}, the real locus $\9{\widehat{T}^{\pr}X}w^{\mu}$ of $\9{\widehat{T}^{\pr}X}w$ is isomorphic to
$\bundle{G^{\sigma}}{\9{L^{\sigma}}w}{n(\9{\widetilde\SSS^{\mu}}w)}$.
In particular,
$$
\9{\widehat{T}^{\pr}X}w^{\mu}=G^{\sigma}(\widetilde\U_{\OO_{\CC}}^{\pr})^{\mu} \quad\text{and}\quad(\widetilde\U_{\OO_{\CC}}^{\pr})^{\mu}=U^{\sigma}\cdot n(\9{\widetilde\SSS^{\mu}}w).
$$
Under the bundle projection $\widetilde{T}^{*}X\to X$, the variety $n(\9{\widetilde\SSS^{\mu}}w)$ maps into $\OO_{\CC}^{\mu}$
and the preimages of the $B^{\sigma}$-orbits $\OO\subseteq\OO_{\CC}^{\mu}$ are open and closed in $n(\9{\widetilde\SSS^{\mu}}w)$, pairwise disjoint, and preserved by $\9{L^{\sigma}}w$.
The latter claim stems from Lemma~\ref{lem:L(R)}
since $[\9Lw,\9Lw]$ fixes $n(\9{\widetilde\SSS}w)$ pointwise.
This implies that the sets $G^{\sigma}\widetilde\U_{\OO}^{\pr}$ are also open and closed in $\9{\widehat{T}^{\pr}X}w^{\mu}$, pairwise disjoint, and collectively cover the whole $\9{\widehat{T}^{\pr}X}w^{\mu}$.
This proves the proposition.
\end{proof}

\subsection{}\label{subsec:T2}

By Remark~\ref{rmk:no-braid}, we already know that the varieties $X$ for which we can not expect an action of the (little) Weyl group belong to the class of the varieties admitting simple roots of type (T2).
In this subsection, we 
describe the objects previously introduced in case $G$ is of semisimple rank $1$ and $X=G/H$ is such that the simple root $\alpha$ is of type (T2).

By assumption, $G=[G,G]\cdot T^{\alpha}$ with $[G,G]\simeq SL_2(\CC)$ or $PGL_2(\CC)$, $T^{\alpha}=(\Ker\alpha)^{\circ}=Z(G)^{\circ}$ , and $H$ projects onto a torus in $G/Z(G)\simeq PGL_2(\CC)$.

We suppose that $B\cap[G,G]$ is the standard Borel subgroup consisting of upper triangular matrices and $T_{\alpha}=T\cap[G,G]$ is the standard diagonal torus.
Furthermore, we choose the base point $x_0$ such that it is contained in $\Omax^{\mu}$, $H$ projects onto $PSO_{1,1}(\CC)$ and $H^{\sigma}$ projects onto $PSO_{1,1}(\RR)$.
The two open $B^{\sigma}$-orbits in $G^{\sigma}x_0$ are $\OO=B^{\sigma}x_0$ and $\OO'=B^{\sigma}n_{\alpha}x_0$, where $n_{\alpha}$ is a representative of $s_{\alpha}$ in $N_{[G,G]}(T)^{\sigma}$ given by the matrix
$$
\begin{pmatrix}
  0 &-1 \\
  1 & 0 \\
\end{pmatrix}.
$$


Let $\ttt_{\perp}$ denote the orthocomplement of $\ttt_0=\h\cap\ttt^{\alpha}$ in $\ttt^{\alpha}$ with respect to the invariant inner product on $\g$. Then we have
$$\h=\left\{
\left(
\begin{pmatrix}
  0 & c \\
  c & 0 \\
\end{pmatrix},z
\right)\;:\;c\in\CC,\ z=z_0+z_{\perp}\in\ttt_0\oplus \ttt_{\perp}\right\},
$$
where $z_{\perp}=z_{\perp}(c)$ depends linearly on $c$. The invariant inner product on $\g$ may be chosen so that its restriction to $\sgl_2$ is the standard trace product (in particular, the above matrix has inner square $2c^2$) and $(z_{\perp},z_{\perp})=2c^2$.

The momentum map $\Phi$ identifies $T_{x_0}^*X$ with $\h^{\perp}=\s\oplus\s^{\perp}$, where $\s=\h^{\perp}\cap\sgl_2(\CC)$, i.e.,
$$
\s=\left\{
\begin{pmatrix}
  a & -b \\
  b & -a \\
\end{pmatrix}\,:\,a,b\in\CC
\right\}
$$
and $\s^{\perp}\subset\h\oplus\ttt_{\perp}$ reads as
\begin{equation*}
\tag{$\bot$}
\s^{\perp}=\left\{
\left(
\begin{pmatrix}
  0 &-c \\
 -c & 0 \\
\end{pmatrix},z
\right)\;:\;c\in\CC,\ z\in\ttt_{\perp}
\right\},
\end{equation*}
where $c=c(z)$ depends linearly on $z$, so that $c(z_{\perp}(c))=c$. The group $H$ acts on $\h^{\perp}$ by pseudo-orthogonal transformations of $\s$ preserving the quadratic form $\det=b^2-a^2$.

\begin{proposition}\label{prop:T2}
The reflection operator corresponding to $\alpha$ interchanges $\OO$ and $\OO'$
whereas the action of $s_{\alpha}$ in $\widetilde{T}^{\pr}X^{\mu}$ interchanges the respective cells of the partition \textup{(P)} only if $\s^{\perp}\subseteq\ttt_{\perp}$.
\end{proposition}

In order to prove the above statement, we proceed to describe the cells of the partition~(P).
The proof of Proposition~\ref{prop:T2} is postponed to the end of this subsection.

Under the natural identification $\aaa^*\simeq\ttt_0^{\perp}=\ttt_{\alpha}\oplus\ttt_{\perp}$,
we denote the elements of $\aaa^*$ as pairs
$$
\zeta=\left(
\begin{pmatrix}
  a & 0 \\
  0 &-a \\
\end{pmatrix},z
\right).
$$

\begin{lemma}
The principal stratum $\aaa^{\pr}$ is given by the open conditions $a\ne0$ and $a\ne\pm c(z)$.
\end{lemma}

\begin{proof}
The first condition is just $\alpha(\zeta)\ne0$. We show that the second condition assures closedness of $Z(\zeta)$.

We have $Z(\zeta)=Tx$, where $x=u_{\alpha}(s)x_0$ is such that $u_{\alpha}(s)\h^{\perp}\ni\zeta$ or, equivalently, 
$$
\h^{\perp}\ni u_{\alpha}(-s)\zeta=
\left(
\begin{pmatrix}
  a & 2as \\
  0 & -a  \\
\end{pmatrix},z
\right).
$$
This yields $s=-c(z)/a$. Now $Z(\zeta)$ is closed in $X$ if and only if $Tu_{\alpha}(s)H$ is closed in $G$ or, equivalently,
$$
T(SL_2)\cdot
\begin{pmatrix}
  1 & s \\
  0 & 1 \\
\end{pmatrix}
\cdot SO_{1,1}
$$
is closed in $SL_2$. Taking
$$
\begin{pmatrix}
  t & 0      \\
  0 & t^{-1} \\
\end{pmatrix}
\in T(SL_2)
\quad\text{and}\quad
\begin{pmatrix}
  p & q \\
  q & p \\
\end{pmatrix}
\in SO_{1,1}
\quad(p^2-q^2=1),
$$
we get the set of matrices of the form
$$
g=\begin{pmatrix}
 tp+tsq & tq+tsp  \\
t^{-1}q & t^{-1}p \\
\end{pmatrix}.
$$
They satisfy equations $\det{g}=g_{11}g_{22}-g_{12}g_{21}=1$ and $g_{12}g_{22}-g_{11}g_{21}=s$, and these are defining equations unless $s=\pm1$. (For $s=\pm1$, the matrices with $g_{11}=\pm g_{12}$, $g_{21}=\mp g_{22}$, and $g_{11}g_{22}=1/2$, not of the above form, satisfy these equations, too.)
\end{proof}

The $G$-span of $\aaa^{\pr}$ consists of the elements $\xi=(\xi_0,z)\in\sgl_2(\CC)\oplus\ttt_{\perp}$ with $\det\xi_0\ne0,-c(z)^2$. The principal stratum $T_{x_0}^{\pr}X$ is thus determined by the inequalities
$$
\begin{vmatrix}
  a & -b-c \\
b-c &   -a \\
\end{vmatrix}
=b^2-a^2-c^2\ne0\qquad\text{ and }\qquad a\ne\pm b,
$$
where the matrix involved on the left hand side is the $\sgl_2$-part $\xi_0$ of $\xi=\Phi(\psi)\in\h^{\perp}$ for $\psi\in T_{x_0}^*X$. The fiber  of $\U_{\OO_{\CC}}$ at $x_0$ is identified (via $\Phi$) with $\h^{\perp}\cap\bb$ and determined by the equality $b=c$ or, equivalently,
$$
\xi_0=\begin{pmatrix}
        a & -2c \\
        0 & -a  \\
      \end{pmatrix}
\qquad\text{ and }\qquad c=c(z).
$$

The polarized cotangent bundle is a homogeneous fiber bundle over $X$ with fiber $\widetilde{T}^*_{x_0}X\subset T^*_{x_0}X\times\aaa^*$ consisting of the pairs $(\psi,\zeta)$ such that
$\Phi(\psi)=(\xi_0,z)$, $\zeta=(\zeta_0,z)$ and $z\in\ttt_{\perp}$ with
\begin{multline*}
\xi_0=
\begin{pmatrix}
  a & -b-c \\
b-c &   -a \\
\end{pmatrix},
\quad
\zeta_0=
\begin{pmatrix}
\tilde{a} &        0  \\
       0  &-\tilde{a} \\
\end{pmatrix},\quad
c=c(z),\quad\text{and}\quad
\det\xi_0=b^2-a^2-c^2=-\tilde{a}^2.
\end{multline*}

The real cotangent space $T_{x_0}^*X^{\mu}$ is identified with $(\h^{\perp})^{\sigma}$ on which $H^{\sigma}$ acts by hyperbolic rotations in the plane $\s^{\sigma}$ as depicted right below
$$
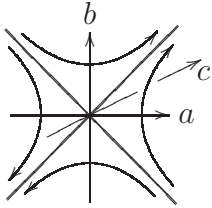
\xymatrix{
\ar@{-}@<-2.4pt>[rrdd] \ar@{-}@<-2.8pt>[rrdd] \POS+<8pt,-7pt>\ar@/_1pc/[rr]+<-11pt,-6pt> \POS+<-5pt,-5pt>\ar@/^1pc/[dd]+<3pt,9pt> & b & \\
\ar[rr] & \relax \POS+<-16pt,-8pt> \ar@{-}+<34pt,17pt> \POS+<42pt,21pt> \ar+<16pt,8pt> \POS+<17pt,3pt> *{c} & a \\
\ar@{-}@<-0.2pt>[rruu] \ar@{-}@<-0.6pt>[rruu] &**{!<0pt,2pt>=<0pt,0pt>} \ar[uu]
& \relax \POS+<-12pt,3pt>\ar@/_1pc/[ll]+<9pt,4pt> \POS+<7pt,4pt>\ar@/^1pc/[uu]+<-5pt,-11pt> \\
}
$$

The above description yields the following lemmas.

\begin{lemma}\label{lem:T2}
The principal stratum $T_{x_0}^{\pr}X^{\mu}$ is given by ${b^2-a^2-c^2\ne0}$, $a\ne\pm b$.
It consists of eight connected components unless $\s^{\perp}\subseteq\ttt_{\perp}$ (i.e., $c=0$ identically on~$\ttt_{\perp}$), in which case it consists of four connected components.
These components, denoted by $\pm1$, $\pm2$, $\pm2'$, $\pm2''$, are represented in the following pictures:
$$
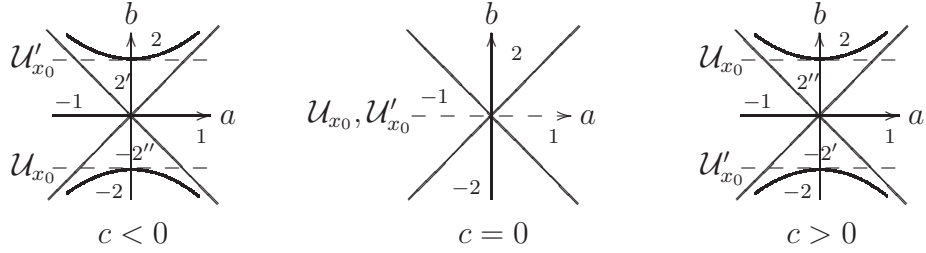
\begin{array}{c@{\qquad}c@{\qquad\quad}c}
\xymatrix{
\ar@{-}@<-2.4pt>[rrdd] \ar@{-}@<-2.8pt>[rrdd] \POS+<9pt,-7pt>\ar@{-}@/_0.8pc/[rr]+<-11pt,-6pt> \POS+<0pt,-0.4pt>\ar@{-}@/_0.8pc/[rr]+<-11pt,-6.4pt> & b & \relax\POS+<-27pt,-9pt>*{\scriptstyle 2} \\
\ar[rr] \POS+<10pt,5pt>*{\scriptstyle -1} \POS+<-13pt,-24.9pt>*{\U_{x_0}} \POS+<7pt,0pt>\ar@{--}+<58pt,0pt> \POS+<-7pt,40.8pt>*{\U'_{x_0}} \POS+<7pt,0pt>\ar@{--}+<58pt,0pt> & \relax\POS+<-3pt,13pt>*{\scriptstyle 2'} \POS+<4.6pt,-27pt>*{\scriptstyle -2''} & a \POS+<-10pt,-8pt>*{\scriptstyle 1} \\
\ar@{-}@<-0.1pt>[rruu] \ar@{-}@<-0.5pt>[rruu] \POS+<9pt,4.2pt>\ar@{-}@/^0.8pc/[rr]+<-11pt,3.2pt> \POS+<0pt,-0.4pt>\ar@{-}@/^0.8pc/[rr]+<-11pt,2.8pt> \POS+<16pt,1pt>*{\scriptstyle -2} & \ar[uu] &  \\
} &
\xymatrix{
\ar@{-}@<-2.4pt>[rrdd] \ar@{-}@<-2.8pt>[rrdd] & b & \relax\POS+<-27pt,-14pt>*{\scriptstyle 2} \\
\ar@{-->}[rr] \POS+<-16pt,0pt>*{\U_{x_0},\U'_{x_0}} \POS+<28pt,7pt>*{\scriptstyle -1} && a \POS+<-13pt,-8pt>*{\scriptstyle 1} \\
\ar@{-}@<-0.1pt>[rruu] \ar@{-}@<-0.5pt>[rruu] \POS+<24pt,6pt>*{\scriptstyle -2} & \ar[uu] &  \\
} &
\xymatrix{
\ar@{-}@<-2.4pt>[rrdd] \ar@{-}@<-2.8pt>[rrdd]\POS+<9pt,-7pt>\ar@{-}@/_0.8pc/[rr]+<-11pt,-6pt> \POS+<0pt,-0.4pt>\ar@{-}@/_0.8pc/[rr]+<-11pt,-6.4pt> & b & \relax\POS+<-27pt,-9pt>*{\scriptstyle 2} \\
\ar[rr] \POS+<10pt,5pt>*{\scriptstyle -1} \POS+<-13pt,15.9pt>*{\U_{x_0}} \POS+<7pt,0pt>\ar@{--}+<58pt,0pt> \POS+<-7pt,-40.8pt>*{\U'_{x_0}} \POS+<7pt,0pt>\ar@{--}+<58pt,0pt> & \relax\POS+<-4pt,13pt>*{\scriptstyle 2''} \POS+<4.4pt,-27pt>*{\scriptstyle -2'} & a \POS+<-10pt,-8pt>*{\scriptstyle 1} \\
\ar@{-}@<-0.1pt>[rruu] \ar@{-}@<-0.5pt>[rruu] \POS+<9pt,4.2pt>\ar@{-}@/^0.8pc/[rr]+<-11pt,3.2pt> \POS+<0pt,-0.4pt>\ar@{-}@/^0.8pc/[rr]+<-11pt,2.8pt> \POS+<16pt,1pt>*{\scriptstyle -2} & \ar[uu] &  \\
} \\
\; c<0 & \;\qquad c=0 & \; c>0
\end{array}
$$
\end{lemma}

The fiber $\U_{x_0}$ of $\U_{\OO}$ at $x_0$ is distinguished by the condition $b=c$;
its location is indicated by some of the dashed lines in the above picture.
The other dashed lines indicate the locus $\U_{x_0}'=n_{\alpha}^{-1}\U_{n_{\alpha}x_0}$, where $\U_{n_{\alpha}x_0}$ is the fiber of $\U_{\OO'}$ at $n_{\alpha}x_0$; it is identified (via $\Phi$) with $(\h^{\perp}\cap\bb^-)^{\sigma}$ and distinguished by the condition $b=-c$.

\begin{lemma}
The locus $(\aaa^{\pr})^{\sigma}$ is distinguished by the open conditions $a\ne0$ and $a\neq\pm c$.
It decomposes into six connected components unless $\s^{\perp}\subseteq\ttt_{\perp}$,
in which case it consists of two components only.
The components are indicated by labels $\pm\mathrm{I}$, $\pm\mathrm{II}'$, $\pm\mathrm{II}''$ in the following picture.
$$
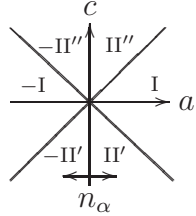
\xymatrix{
\relax\POS+<3.5pt,-8pt> \ar@{-}@<-0.2pt>[rrdd]+<-5pt,4pt> \ar@{-}@<0.2pt>[rrdd]+<-5pt,4pt> \POS+<19pt,-6.6pt>*{\scriptstyle-\mathrm{II}''} & c & \relax\POS+<-25pt,-14pt>*{\scriptstyle\mathrm{II}''} \\
\ar[rr] \POS+<12pt,6pt>*{\scriptstyle-\mathrm{I}} && a \POS+<-12pt,7pt>*{\scriptstyle\mathrm{I}} \\
\relax\POS+<3.5pt,4pt> \ar@{-}@<-0.2pt>[rruu]+<-8pt,-8pt> \ar@{-}@<0.2pt>[rruu]+<-8pt,-8pt> \POS+<20pt,9.5pt>*{\scriptstyle-\mathrm{II}'} \POS+<12pt,-18pt> *{n_{\alpha}} \POS+<-12pt,10pt> \ar@{<->}+<20pt,0pt> & \ar@{-}@<-0.2pt>[uu]+<0pt,-8pt> \ar[uu] \ar@{-}@<0.2pt>[uu]+<0pt,-8pt> & \relax\POS+<-27pt,14pt>*{\scriptstyle\mathrm{II}'} \\
}$$
The action of $s_{\alpha}$ is indicated by the double-headed arrow.
\end{lemma}

\begin{proof}[Proof of Proposition~\ref{prop:T2}]
Consider the covering $\widetilde{T}^{\pr}_{x_0}X^{\mu}\to T^{\pr}_{x_0}X^{\mu}$.
It is a two-fold covering: the fiber over $\psi\in T^{\pr}_{x_0}X^{\mu}$ consists of two points $(\psi,\zeta)$ such that the diagonal entries $\pm\tilde{a}$ of the $\ttt_{\alpha}$-part $\zeta_0$ of $\zeta$ are the eigenvalues of the $\sgl_2$-part $\xi_0$ of $\xi=\Phi(\psi)$.
To determine these eigenvalues, one may move $\xi_0$ to the locus $\{b=\pm c\}$ (i.e., to an upper or lower triangular shape) by a hyperbolic rotation and take for $\tilde{a}$ or $-\tilde{a}$ the abscissa of the resulting point.

The components $\pm2$ of $T^{\pr}_{x_0}X^{\mu}$ (in which $\xi_0$ has pure imaginary eigenvalues) do not belong to the image of $\widetilde{T}^{\pr}_{x_0}X^{\mu}$ and the preimage of any other component consists of two connected components of $\widetilde{T}^{\pr}_{x_0}X^{\mu}$ differing by the sign of $\tilde{a}$.
We denote them by pairs of labels, the first, resp.\ second, one being the label of the component of $T^{\pr}_{x_0}X^{\mu}$, resp.\ of $(\aaa^{\pr})^{\sigma}$, containing $\psi$, resp.~$\zeta$. There are twelve components $(\pm1,\pm\text{I})$, $(2',\pm\text{II}')$, $(2'',\pm\text{II}'')$, $(-2',\pm\text{II}'')$, $(-2'',\pm\text{II}')$, unless $\s^{\perp}\subseteq\ttt_{\perp}$, in which case there are only the four components $(\pm1,\pm\text{I})$.

Since $H^{\sigma}$ acts on $\widetilde{T}^{\pr}_{x_0}X^{\mu}$ via the connected Lie group of hyperbolic rotations,  $H^{\sigma}$ preserves the components of $\widetilde{T}^{\pr}_{x_0}X^{\mu}$, and their $G^{\sigma}$-spans are the \emph{$G^{\sigma}$-connected components} (i.e., minimal open and closed $G^{\sigma}$-stable subsets) of $\widetilde{T}^{\pr}X^{\mu}$ lying over $G^{\sigma}x_0=G^{\sigma}\OO=G^{\sigma}\OO'$.
We denote them by the same pairs of labels. The action of $s_{\alpha}$ permutes these $G^{\sigma}$-connected components by changing the sign of the second label.

The fiber $\widetilde{\U}_{x_0}$ in $\widetilde{\U}_{\OO}$ is the set of pairs $(\psi,\zeta)$
such that $b=c$ and $a=\tilde{a}$, with the above notation.
It follows that $\widetilde{\U}_{\OO}$ meets the components $(1,\text{I})$, $(-1,-\text{I})$, $(2'',\pm\text{II}'')$, and
$(-2'',\pm\text{II}')$.
Similarly, the set $\widetilde{\U}'_{x_0}=n_{\alpha}^{-1}\widetilde{\U}_{n_{\alpha}x_0}$ (where $\widetilde{\U}_{n_{\alpha}x_0}$ is the fiber of $n_{\alpha}x_0$ in $\widetilde{\U}_{\OO'}$) consists of $(\psi,\zeta)$
such that $b=-c$ and $a=-\tilde{a}$ (because the adjoint action of $n_{\alpha}$ preserves $b$ and changes the signs of $a$ and $c$).
It follows that $\widetilde{\U}_{\OO'}$ meets the other components of $\widetilde{T}^{\pr}X^{\mu}$ over $G^{\sigma}x_0$, namely the components $(1,-\text{I})$, $(-1,\text{I})$, $(2',\pm\text{II}')$, and $(-2',\pm\text{II}'')$ .

From the above discussion, we deduce that $G^{\sigma}\widetilde\U_{\OO}^{\pr}$ is the union of the $G^{\sigma}$-connected components
$(1,\text{I})$, $(-1,-\text{I})$, $(2'',\pm\text{II}'')$, and $(-2'',\pm\text{II}')$,
whereas $G^{\sigma}\widetilde\U_{\OO'}^{\pr}$ is the union of the components $(1,-\text{I})$, $(-1,\text{I})$, $(2',\pm\text{II}')$, and $(-2',\pm\text{II}'')$.
By Lemma~\ref{lem:T2}, the last four of the six components in each case are not present whenever $\s^{\perp}\subseteq\ttt_{\perp}$.
Furthermore, the components $(\pm1,\pm\text{I})$ move to a different cell but the other components remain in the same cell.
This concludes the proof of Proposition~\ref{prop:T2}.
\end{proof}

\subsection{}

Recall the definition of types for $B^\sigma$-orbits of $X(\RR)$; see Section~\ref{sec:min&refl}.

\begin{theorem}
\label{thm:real-Knop-action}
Suppose $X$ is a spherical homogeneous $G$-variety defined over $\RR$ with no
simple roots of type \textup{(T2)} w.r.t.\ real
Borel orbits in $\Bmax(X^{\mu})$, or a symmetric space (defined over $\RR$).
Then the bijection $\OO\mapsto G^{\sigma}\widetilde\U_{\OO}^{\pr}$
defined by the partition~\eqref{eqn:partition} is $\widetilde W$-equivariant.
In particular, the $\widetilde W$-action on the set $\Bmax(X^\mu)$ 
factors through $W$.

Moreover, the resulting action of $W$ on $\Bmax(X^\mu)$ agrees with the $W$-action on $\Bmax(X)$.
\end{theorem}

The proof of Theorem~\ref{thm:real-Knop-action} is postponed to the end of this section.

\begin{corollary}\label{cor:thm-real}
Suppose $X$ is a spherical homogeneous $G$-variety defined over $\RR$ 
and satisfying the assumption of Theorem~\ref{thm:real-Knop-action}.
Then the little Weyl group of $X$ acts on the set of open $B(\RR)$-orbits in $X(\RR)$
and the orbits of this action are in bijective correspondence with the $G(\RR)$-orbits in $X(\RR)$.
\end{corollary}

\begin{proof}
This follows readily from Theorem~\ref{thm:real-Knop-action}, Theorem~\ref{thm:B-orb&comp-C}, and \S\ref{subsec:min&refl-R}.
\end{proof}

\begin{remark}\label{rmk:N_X-act-real}
As in the complex case, by Corollary~\ref{cor:oftheproof} we may define an ``action'' of $N_X^{\sigma}$ on $Z^{\mu}$ for $Z=Z(\zeta)$ ($\zeta\in(\aaa^{\pr})^{\sigma}$) by putting $n\star x=x'\in Z^{\mu}$ such that $ux'=nx$ for a unique $u\in\Ru{P}^{\sigma}$. By Theorem~\ref{thm:real-Knop-action}, if $x\in\OO\in\Bopen(X^{\mu})$, then $n\star x\in\9\OO{w}$, where $w\in W_X$ is represented by $n\in N_X^{\sigma}$. It follows from Corollary~\ref{cor:thm-real} and \S\ref{sec:open-Borel-orb} that this $\star$ operation descends to an action of $W_X$ on $Z^{\mu}/T^{\sigma}$.
\end{remark}

\begin{question}\label{que:N_X-act-real}
Does $\star$ define a true action of $N_X^{\sigma}$ on $Z^{\mu}$?
\end{question}

For symmetric spaces, the answer is affirmative, see \S\ref{ex:symm}.

\begin{proof}[Proof of Theorem~\ref{thm:real-Knop-action}]
We adapt the approach conducted by Knop in the proof of Lemma 6.6 in~\cite{knop:action} to reduce the problem to the case of groups of semisimple rank $1$.

Recall Remark~\ref{rmk:rk-reduction} and the notation set up therein with $S=\{\alpha\}$.
Take $\OO\in\Bmax(X^\mu)$ and let $Y=P_{\alpha}\OO_{\CC}$. The 
geometric quotient $V=Y/\Ru{(P_{\alpha})}$ is 
a spherical homogeneous $L_{\alpha}$-variety and,
as already observed, the $B$-orbits in $Y$ are the preimages under the quotient map $\pi:Y\to V$ of the $B_\alpha$-orbits in $V$; this correspondence preserves the configuration type in the sense of~\S\ref{subsec:min&refl-C}. Furthermore, the rank of the orbits under consideration is also preserved.

Moreover, by Remark~\ref{rmk:rk-red-real}, $P_{\alpha}^{\sigma}\OO$ is open and closed in $Y^{\mu}$ and $\OO$ is the preimage under $\pi|_{Y^{\mu}}$ of a $B_{\alpha}^{\sigma}$-orbit $\Q\subset V^{\mu}$.
Again, the correspondence between the $B^{\sigma}$-orbits in $P_{\alpha}^{\sigma}\OO$ and the $B_{\alpha}^{\sigma}$-orbits in $L_{\alpha}^{\sigma}\Q$ preserves the configuration types in the sense of~\S\ref{subsec:min&refl-R}.

Let $\V\subset T^*X$ denote the conormal bundle of the foliation of the $\Ru{(P_{\alpha})}$-or\-bits in $Y$ and $\N\subset\V$ be the conormal bundle of $Y$.
Then $\V/\N\subseteq T^*Y$ is the pullback of $T^*V$ along the quotient map $\pi:Y\to V$ and $\V/\N\to T^*V$ is the quotient by $\Ru{(P_{\alpha})}$. We get a commutative diagram
$$\xymatrix{
Y \ar[d]_{\pi} & \V \ar[l] \ar[d] \ar[rr]^{\text{momentum}} && **[r] \Ru{(\p_{\alpha})}^{\perp}\simeq\p_{\alpha} \ar[d]^{\text{projection}} \ar[r] & **[r] \g^*\quot{G}\simeq\ttt^*/W \\
V        & T^*V \ar[l]      \ar[rr]^{\text{momentum}} && **[r] \llll_{\alpha}^*\simeq\llll_{\alpha} \ar[r] & **[r] \llll_{\alpha}^*\quot{L_{\alpha}}\simeq\ttt^*/W_{\alpha}, \ar@<-1ex>[u]
}$$
where $W_{\alpha}=\{1,s_{\alpha}\}$ is the Weyl group of $L_{\alpha}$. Note that $\U_{\OO_{\CC}}$ is the preimage of $\U_{\Q_{\CC}}\subseteq T^*V$ in $\V$ and there is a commutative diagram
$$\xymatrix{
\U_{\OO_{\CC}} \ar[dd] \ar[rr]^{\text{momentum}} && **[r] \bb=\uu\oplus\ttt \ar[dd]^{\text{projection}} \ar[rrrr] \POS+RD \ar[rrd]_-{\text{projection}} &&&& \ttt^*/W\hphantom{{}_{\alpha}.} \\
&&&& **[r] \ttt\simeq\ttt^* \POS[]="pos"+RU \ar[rru] \POS"pos"+RD \ar[rrd] && \\
\U_{\Q_{\CC}}  \ar[rr]^{\text{momentum}} && **[r] \bb_{\alpha}=\uu_{\alpha}\oplus\ttt \ar[rrrr] \POS+RU \ar[rru]^-{\text{projection}} &&&& \ttt^*/W_{\alpha}. \ar@<1ex>[uu]
}$$

The real locus $\V^{\mu}$ is the conormal bundle of the foliation of $\Ru{(P_{\alpha})}^{\sigma}$-orbits in $Y^{\mu}$ and $\U_{\OO}$ is the preimage of $\U_{\Q}$ in $\V^{\mu}$.

It follows from the commutative diagram right above that $\widetilde\U_{\OO_{\CC}}$ is the preimage of $\widetilde\U_{\Q_{\CC}}$ under the natural $W_\alpha$-equivariant map
$$\widetilde{T}^*X\supset\V\mathbin{\underset{\ttt^*/W_{\alpha}}\times}\ttt^*\longrightarrow T^*V\mathbin{\underset{\ttt^*/W_{\alpha}}\times}\ttt^*=\widetilde{T}^*V.$$
The induced bijective map between the sets of irreducible components of these fiber products, which are the closures of the $P_{\alpha}\widetilde\U_{\OO_{\CC}}$'s ($\OO_\CC\in\Bmax(Y)$) and of the $L_{\alpha}\widetilde\U_{\Q_{\CC}}$'s ($\Q_\CC\in\Bmax(V)$), respectively, is $W_\alpha$-equivariant, too. Consequently, the $W_\alpha$-actions on the set of the $G\widetilde\U_{\OO_{\CC}}^{\pr}$'s and on the set of the $L_{\alpha}\widetilde\U_{\Q_{\CC}}^{\pr}$'s match.

Passing to the real loci, we readily see that the 
sets of the $G^\sigma$-connected components of the $G^\sigma\widetilde\U_{\OO}^{\pr}$'s ($\OO\in\Bmax(Y^\mu)$) and 
of the $L_{\alpha}^\sigma$-connected components of the $L_{\alpha}^{\sigma}\widetilde\U_{\Q}^{\pr}$'s ($\Q\in\Bmax(V^\mu)$) are in bijective correspondence and the $W_\alpha$-actions on these two sets
match. This together with the fact that the natural bijective map between the sets $\B(Y^\mu)$ and $\B(V^\mu)$ is $W_\alpha$-equivariant  yields the announced reduction to the semisimple rank $1$.

For the rest of the proof, except for the last paragraph, let $G=P_{\alpha}=L_{\alpha}$ be of semisimple rank $1$. Consider the possible configurations of $B^{\sigma}$-orbits in $P_{\alpha}^{\sigma}\OO$.
We may assume that $\OO_{\CC}=\Omax$, i.e., $\OO$ is an open $B^{\sigma}$-orbit in $X^{\mu}$.

\subsubsection*{Types \textup{(P)}, \textup{(T0)}, \textup{(T1)}, \textup{(N0)}, and \textup{(N1)}}
Here $\OO$, $\OO_{\CC}$ and $G\widetilde\U_{\OO_{\CC}}^{\pr}$ are preserved by $s_{\alpha}$.
As for $G^{\sigma}\widetilde\U_{\OO}^{\pr}$, $s_{\alpha}$ maps it to an open and closed subset of the real locus of $G\widetilde\U_{\OO_{\CC}}^{\pr}$ which lies over $G^{\sigma}\OO$, i.e., to $G^{\sigma}\widetilde\U_{\OO}^{\pr}$ itself, because $\OO$ is the only $B^{\sigma}$-orbit both in $G^{\sigma}\OO$ and in $\Bmax(X^\mu)$.

\subsubsection*{Type \textup{(U)}}
Here $G^{\sigma}\OO=\OO\cup\OO'$ with $\OO'$ being a $B^{\sigma}$-orbit of codimension~1 and $s_{\alpha}\OO=\OO'$.
Moreover, $X=\OO_{\CC}\cup\OO'_{\CC}$ and $s_{\alpha}\OO_\CC=\OO'_\CC$.
The two connected components $G\widetilde\U_{\OO_{\CC}}^{\pr}$ and $G\widetilde\U_{\OO'_{\CC}}^{\pr}$ of $\smash[t]{\widetilde{T}^{\pr}X}$ are interchanged by $s_{\alpha}$. Recall that $G\smash[t]{\widetilde\U_{\OO_{\CC}}^{\pr}}=G\smash[t]{\widetilde\SSS}$ and $\smash[t]{\widetilde\SSS(\zeta)}$ projects onto $Z(\zeta)$ for any $\zeta\in\aaa^{\pr}$. (In fact, it is easy to see that all slices $Z(\zeta)$ are the same in this case.)

Also, $G^{\sigma}\widetilde\U_{\OO}^{\pr}=G^{\sigma}(\widetilde\SSS^{\mu}\cap\widetilde\U_{\OO})$
hence $s_\alpha$ maps $G^{\sigma}\widetilde\U_{\OO}^{\pr}$ to $G^{\sigma}n_{\alpha}\9{(\widetilde\SSS^{\mu}\cap\widetilde\U_{\OO})}{s_{\alpha}}$, with $n_{\alpha}\in N_G(T)^{\sigma}$ a representative of $s_{\alpha}$.
Since $\widetilde\SSS(\zeta)^{\mu}\cap\widetilde\U_{\OO}$ projects onto $Z(\zeta)^{\mu}\cap\OO$ for any $\zeta\in(\aaa^{\pr})^{\sigma}$, $n_{\alpha}\9{(\widetilde\SSS(\zeta)^{\mu}\cap\widetilde\U_{\OO})}{s_{\alpha}}$ projects onto $n_{\alpha}(Z(\zeta)^{\mu}\cap\OO)$.
By Corollary~\ref{cor:oftheproof}, $n_{\alpha}(Z(\zeta)^{\mu}\cap\OO)\subset\OO'_{\CC}$ and in turn
$n_{\alpha}(Z(\zeta)^{\mu}\cap\OO)\subset\OO'_\CC\cap G^{\sigma}\OO=\OO'$.
It follows that $n_{\alpha}\9{(\widetilde\SSS^{\mu}\cap\widetilde\U_{\OO})}{s_{\alpha}}\subset\widetilde\U_{\OO'}$ hence $s_{\alpha}$ sends $G^{\sigma}\widetilde\U_{\OO}^{\pr}$ to $G^{\sigma}\widetilde\U_{\OO'}^{\pr}$.

\subsubsection*{Type \textup{(N2)}}
Here $G^{\sigma}\OO$ contains two open $B^{\sigma}$-orbits $\OO$ and $\OO'$ interchanged by $s_{\alpha}$ and both lying in $\Omax$.
Besides, $\Omax$ is preserved by $s_{\alpha}$ and so is $G\widetilde\U_{\Omax}^{\pr}$.
The setting is similar to type (T2) in \S\ref{subsec:T2} except that $H^{\sigma}$ projects onto $PO_{1,1}(\RR)$ instead.
In particular, the (co)adjoint action of an element $r\in H^{\sigma}$ with image
$$
\begin{bmatrix}
 -1 & 0 \\
  0 & 1 \\
\end{bmatrix}
$$
in $PO_{1,1}(\RR)$ preserves $\h,\h^{\perp},\s,\s^{\perp}$ and changes the sign of $c$ leaving $z$ fixed in ($\perp$). It follows that $c=0$ identically on $\ttt_{\perp}$.
So, 
arguing as in the proof of Proposition~\ref{prop:T2},
we get that $s_{\alpha}$ sends $G^{\sigma}\widetilde\U_{\OO}^{\pr}$ to $G^{\sigma}\widetilde\U_{\OO'}^{\pr}$. 

\bigskip\noindent
\emph{Type} (T2) is 
excluded from our consideration 
unless $X=G/H$ is a symmetric space. In the latter case, recall the notation set up in~\S\ref{subsec:symm}.
The momentum map identifies $T_{x_0}^*X$ with
$$
\h^{\perp}=\ttt_1\oplus\bigoplus_{\alpha\ne\theta(\alpha)}\h^{\perp}_{\alpha,\theta(\alpha)},
$$
where $\h^{\perp}_{\alpha,\theta(\alpha)}=\langle e_{\alpha}-e_{\theta(\alpha)}\rangle$ is ``antidiagonally'' embedded in $\g_{\alpha}\oplus\g_{\theta(\alpha)}$. In particular, $\aaa=\ttt_1\subset\h^{\perp}$.
Hence $Z=Tx_0=Z(\zeta)$ is a flat for any $\zeta\in\aaa^{\pr}$ and $nZ\subset\9{\Omax}w$ for any $w\in W$ represented by $n\in N_G(T)$. Therefore each orbit $\OO\in\Bmax(X^{\mu})$ contains a point $x=nx_0$ and, for a simple root $\alpha$ of type (T2) w.r.t.\ $\OO$, $L_{\alpha}$ is stable under $\theta_x$ and $V=P_{\alpha}x/\Ru{(P_{\alpha})}$ is a symmetric space for $L_{\alpha}$, by Proposition~\ref{prop:root-symm}. Thus we may again assume that the acting group $G$ is of semisimple rank~1 and $H$ is a symmetric subgroup of $G$, in which case $\s^\perp\subset\ttt_\perp$, in the notation of \S\ref{subsec:T2}.
We conclude the proof by Proposition~\ref{prop:T2}.
\end{proof}

\section{Examples}
\label{sec:examples}

\subsection{Quadratic forms}
\label{ex:quad}

Quadratic forms of fixed rank $r$ in $n$ variables form a homogeneous variety $X$ for the group $G=GL_n$. For the base point $x_0$ we may take the standard quadratic form $z_1^2+\dots+z_r^2$ of rank $r$ in the variables $z_1,\dots,z_n$. Its stabilizer $H$ has a Levi decomposition $H=K\ltimes\Ru{H}$ with $K=O_r\times GL_{n-r}$ and $\Ru{H}$ being the Abelian group of unipotent linear transformations acting on the variables as follows: $z_i\mapsto z_i$ for $i\le r$ and $z_j\mapsto z_j+\sum_{i\le r}c_{ji}z_i$ for $j>r$. The space of quadratic forms $X$ is equipped with the natural real structure given by complex conjugation of the coefficients of a quadratic form. This real structure is compatible with the natural split real structure on the acting group $GL_n$.

Let $B\subset GL_n$ be the standard Borel subgroup of upper triangular matrices and $T\subset B$ be the standard diagonal torus. Denote by $\eps_1,\dots,\eps_n$ the eigenweights of the standard basis vectors in $\CC^n$ (i.e., the diagonal entries of $T$). The simple roots are $\alpha_i=\eps_i-\eps_{i+1}$ ($i=1,\dots,n-1$).

It is easy to see that $Bx_0=\Omax$ is open in $X$ and therefore $X$ is a spherical homogeneous space. Precisely, $\Omax$ is given by the inequalities $\Delta_1\ne0$, \dots, $\Delta_r\ne0$, where $\Delta_k$ denotes the $k$-th upper left corner minor of the matrix of a quadratic form. Each $\Delta_k$ is an eigenfunction for $B$ with eigenweight $2\omega_k$, where $\omega_k=\eps_1+\dots+\eps_k$, and $\Delta_1,\dots,\Delta_r$ generate the multiplicative group $\CC(X)^{\times}/\CC^{\times}$.

The slice $Z=Tx_0$ consists of the quadratic forms $a_1z_1^2+\dots+a_rz_r^2$ with $a_1\ne0$, \dots, $a_r\ne0$. Moving a point in $\Omax$ to $Z$ by the action of $U$ is just the Gram--Schmidt orthogonalization. Note that $\aaa$ is identified with the diagonal toral subalgebra in $\gl_r$ embedded in $\gl_n$ at the upper left corner and therefore $\aaa\subset\h^{\perp}$, whence $Z=Z(\zeta)$for any $\zeta\in\aaa^{\pr}$.

Since $X$ is parabolically induced from the symmetric space $Y=GL_r/O_r$, we get $\Sigma_X=\Sigma_Y=\{2\alpha_1,\dots,2\alpha_r\}$ (see e.g. \cite[Prop.\,20.13 and Thm.\,26.25]{tim}).

It is an exercise in linear algebra to show that the Borel orbits on $X$ are distinguished by the ranks of the upper left corner submatrices in the matrix of a quadratic form and may be represented by quadratic forms of the shape $\sum_ia_iz_i^2+\sum_{j<k}a_{jk}z_jz_k$ with $a_i\ne0$, $a_{jk}\ne0$, where the summation is over a set of $r$ pairwise distinct numbers $i,j,k$. By means of the action of $T$ one can put such a quadratic form to the canonical shape with all $a_i$ and $a_{jk}$ equal to $1$. Thus the set $\B(X)$ may be parameterized by combinatorial patterns of the form
$$\text{\Huge{$\strut$}}
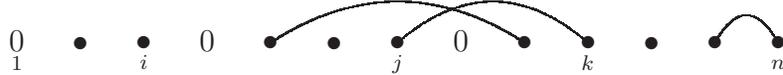
\xymatrix@!0{
\underset{1}{0} & \bullet &
*=0{\smash[b]{\underset{i}{\bullet}}} & 0 & *=0{\bullet} \ar@{-}@/^3ex/[rrrr] & \bullet & *=0{\smash[b]{\underset{j}{\bullet}}} \ar@{-}@/^3ex/[rrr] & 0 & *=0{\bullet} & *=0{\smash[b]{\underset{k}{\bullet}}} & \bullet &
*=0{\bullet} \ar@{-}@/^2ex/[r] & *=0{\smash[b]{\underset{n}{\bullet}}}
}$$
where the dots indicate the positions of $i,j,k$ in the range of $1,\dots,n$, zeroes indicate the other positions, and the arcs connect the positions $j$ and $k$ whenever $a_{jk}\ne0$. On the real locus, the signs of $a_i$ are preserved by the action of $T^{\sigma}$, so that $a_i=\pm1$, $a_{jk}=1$ in the canonical form, and $\B(X^{\mu})$ is parameterized by the \emph{signed patterns} of the form
$$\text{\Huge{$\strut$}}
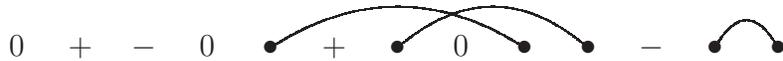
\xymatrix@!0{
0 & {+} & {-} & 0 & *=0{\bullet} \ar@{-}@/^3ex/[rrrr] & {+} & *=0{\bullet} \ar@{-}@/^3ex/[rrr] & 0 & *=0{\bullet} & *=0{\bullet} & {-} & *=0{\bullet} \ar@{-}@/^2ex/[r] & *=0{\bullet}
}$$

The types of simple roots with respect to $B$- or $B^{\sigma}$-orbits can be described as follows.

For a simple root $\alpha=\alpha_i$, the minimal parabolic subgroup $P_{\alpha}$ consists of the matrices with zeroes below the diagonal, except for the entry $(i+1,i)$. The Levi subgroup $L_{\alpha}$ acts by arbitrary linear transformations of the variables $z_i,z_{i+1}$ multiplying other variables by scalars.

For any $B^{\sigma}$-orbit $\OO$ corresponding to a signed pattern as above and represented by a canonical form $x$, the type of $\alpha$ with respect to $\OO$ or $\OO_{\CC}=Bx$ is determined by the group $R_x$ which is the image of $(P_{\alpha})^{\sigma}_x$ in $PGL_2(\RR)$ viewed as a group of projective linear transformations of $z_i,z_{i+1}$. The following table (in which we use the notation of \S\ref{subsec:min&refl-R}) describes the possible types depending on the entries in the pattern at the positions $i,i+1$:
$$
\renewcommand{\arraystretch}{1.2}
\begin{array}{|c|c|c|c|}
\hline
\text{Pattern} & R_x & \text{Type} & \text{Orbit} \\
\hline
{+}\;{+} \quad \text{or} \quad {-}\;{-} & PO_2(\RR) & \text{(N0)} & \text{single} \\
\hline
{+}\;{-} \quad \text{or} \quad {-}\;{+} & PO_{1,1}(\RR) & \text{(N2)} & \text{open} \\
\hline
\text{\Huge{$\strut$}}\xymatrix@!0{
*=0{\bullet} \ar@{-}@/^2ex/[r] & *=0{\bullet}
} & T(PSL_2(\RR))\rtimes\langle r\rangle_2 & \text{(N2)} & \text{closed} \\
\hline
0\;0 & PGL_2(\RR) & \text{(P)} & \text{single} \\
\hline
\text{other} & B(PGL_2(\RR)) \text{ or } B^-(PGL_2(\RR)) & \text{(U)} & \text{closed or open} \\
\hline
\end{array}
$$
The last column indicates the location of $\OO$ in $P_{\alpha}^{\sigma}\OO$ or of $\OO_{\CC}$ in $P_{\alpha}^{\sigma}\OO_{\CC}$. In all cases, $\9\OO{s_{\alpha}}=B^{\sigma}n_{\alpha}x$ and $\9{\OO_{\CC}}{s_{\alpha}}=Bn_{\alpha}x$, where $n_{\alpha}\in[L_{\alpha},L_{\alpha}]$ acts as $z_1\mapsto-z_2$, $z_2\mapsto z_1$.

It follows that the Weyl group $W=S_n$ acts on $\B(X)$ and $\B(X^{\mu})$ by permutations of entries in the patterns. In particular, Question~\ref{que:W-act-real} gets resolved in this case.

From the above discussion, it is also clear that the Borel orbits of maximal rank are characterized by the condition that the respective patterns contain no arcs and open $B^{\sigma}$-orbits are given by the $r$-tuples of signs $\pm$ (followed by $n-r$ zeroes). The (very) little Weyl group $W_X=\vlW_X=S_r$ acts on these tuples by permutations and the orbits for this action are parameterized by numbers of pluses and minuses, i.e., by the inertia indices of quadratic forms. Thus Corollaries \ref{cor:parametrization} and \ref{cor:thm-real} in this case are nothing but the law of inertia. Questions \ref{que:N_X-action} and \ref{que:N_X-act-real} are answered affirmatively in this example, because $N_X$ preserves $Z$ and acts on it by permuting the coefficients $a_i$ ($i=1,\dots,r$) and multiplying them by nonzero squares.

\subsection{Symmetric spaces}
\label{ex:symm}

Let $X=G/H$ be a symmetric space defined over $\RR$.  We shall use the notation of \S\ref{subsec:symm}. In this notation, the involutions $\theta$ and $\sigma$ commute since $\h=\g^{\theta}$ and $\h^{\perp}=\g^{-\theta}$ are preserved by $\sigma$. (Here $\g^{\pm\theta}$ are the $(\pm1)$-eigenspaces of $\theta$ in $\g$.) We may choose the tori $T_0,T_1$ to be defined and split over $\RR$ (see e.g.\ \cite[2.1]{CFT}).

The base point $x_0$ lies in $\Omax^{\mu}$ and $Z:=Tx_0=T_1x_0$ is a slice defined over $\RR$. Note that $Z=Z(\zeta)$, $\forall\zeta\in\aaa^{\pr}$, because $\aaa=\ttt_1\subset\h^{\perp}=\Phi(T^*_{x_0}X)$. The little Weyl group is $$W_X=W(\aaa)=N_G(T_1)/L=N_H(T_1)/L_0,$$ where $L=Z_G(T_1)=T_1\cdot L_0$ and $L_0=L\cap H$ (see e.g.\ \cite[Prop.\,26.19 and Cor.\,26.26]{tim}). In particular, we have $$N_X=L\cdot N_H(T_1)=T_1\cdot N_H(T_1).$$

Consider the following subgroup of $G$:
$$N_0=\{n\in N_H(T_1)\mid n\cdot\sigma(n)^{-1}\in T_1\cap H\}$$
(cf.\ \cite[4.1]{CFT}). By decomposing any $n\in N_X$ as $n=tn_0$, $t\in T_1$, $n_0\in N_H(T_1)$, it is straightforward to check that $n\in N_X^{\sigma}$ if and only if
$$n_0\cdot\sigma(n_0)^{-1}=t^{-1}\cdot\sigma(t)\in T_1\cap H,$$
i.e., $n_0\in N_0$. Conversely, for any $n_0\in N_0$ one finds $t\in T_1$ such that the above equality holds (by taking e.g.\ a square root of $\sigma(n_0)n_0^{-1}$ in $T_1$), whence $n=tn_0\in N_X^{\sigma}$. This $t=t(n_0)$ is uniquely defined modulo $T_1^{\sigma}$ and the map $N_0\to T_1/T_1^{\sigma}$ given by $n_0\mapsto t(n_0)T_1^{\sigma}$ is a cocycle, i.e., $t(n_1n_2)=t(n_1)\cdot\9{t(n_2)}{n_1}\cdot T_1^{\sigma}$, $\forall n_1,n_2\in N_0$, where $\9{}n$ stands for conjugation by $n$.

Clearly, $Z^{\mu}$ is preserved by $N_X^{\sigma}$ and, by Corollary~\ref{cor:thm-real} and Remark~\ref{rmk:N_X-act-real}, the $G^{\sigma}$-orbits intersect $Z^{\mu}$ in $N_X^{\sigma}$-orbits. The latter orbits correspond bijectively to the orbits of the action of $N_0$ on $Z^{\mu}/T_1^{\sigma}\simeq\0Z2/\0{T_1}2$ given by $x\mapsto t(n_0)n_0x$ ($x\in Z^{\mu}$, $n_0\in N_0$). The resulting bijective correspondence between the $G^{\sigma}$-orbits on $X^{\mu}$ and the $N_0$-orbits on $\0Z2/\0{T_1}2$ is the main result of \cite{CFT} (Theorem~4.6), which was obtained in loc.~cit.\ by using the Galois cohomology technique.

\subsection{}
\label{ex:ord.pairs}

Consider the homogeneous space $X=G/H$, where $G=\Spin_{2n+1}$ and $H=SL_n\cdot\CC^{\times}$ is the Levi part of the maximal parabolic subgroup of $G$ corresponding to the short simple root. A model for $X$ is the space of pairs $(V,V')$ of transversal maximal isotropic subspaces in the quadratic space $\CC^{2n+1}$. We assume that the inner product on $\CC^{2n+1}$ in the standard basis $(e_1,\dots,e_n,e_0,e_{-n},\dots,e_{-1})$ is given by $(e_i,e_{-i})=1$ ($i=0,\pm1,\dots,\pm n$) and $(e_i,e_j)=0$ otherwise.

The standard real structure on $\CC^{2n+1}$ (given by complex conjugation of coordinates in the above basis) is compatible with the inner product and induces a split real structure on $SO_{2n+1}$ (given by complex conjugation of matrix entries), which lifts to a split real structure on $G$, and a $G$-equivariant real structure on $X$ given by $\mu(V,V')=(\overline{V},\overline{V'})$, where the bar indicates complex conjugation in $\CC^{2n+1}$.

The real locus $X^{\mu}$ is identified with the space of pairs of transversal maximal isotropic subspaces in $\RR^{2n+1}$, on which the inner product restricts to a quadratic form of signature $(n+1,n)$. The group $G^{\sigma}=\Spin_{n+1,n}(\RR)$ acts on $X^{\mu}$ transitively via its quotient $SO_{n+1,n}(\RR)^{\circ}$.

It will be convenient to replace $G$ and $G^{\sigma}$ with their quotient groups $SO_{n+1,n}(\CC)$ and $SO_{n+1,n}(\RR)^{\circ}$ acting effectively on $X$ and $X^{\mu}$, respectively. All Borel and parabolic subgroups, maximal tori, and their real loci are respectively replaced by their images in $SO_{n+1,n}$.

The variety $X$ is wonderful and the invariants introduced in \S\S\ref{sec:open-Borel-orb} and \ref{subsec:recalls} are known: see e.g.\ case (32) in \S3 of \cite{wonder.red}. Let us describe them here.

Take for $B$ and $T$ the standard Borel subgroup and the standard maximal torus in $SO_{n+1,n}$, given by upper triangular and diagonal special orthogonal matrices in the above basis, respectively. Let $\pm\eps_i$ denote the eigenweights of $e_{\pm i}$ with respect to $T$ ($i=1,\dots,n$). Then $(\eps_1,\dots,\eps_n)$ is the orthonormal basis of $\Xi(X)$ with respect to the (properly scaled) invariant inner product, the simple roots of $G$ are $\alpha_i=\eps_i-\eps_{i+1}$ with $i<n$ (long roots) and $\alpha_n=\eps_n$ (short root), and the character lattice of $T$ is an index 2 extension of $\Xi(X)$ spanned by $(\eps_1+\dots+\eps_n)/2$.

The set of spherical roots is $\Sigma_X=\{\gamma_1,\dots,\gamma_n\}$, where
\begin{align*}
\gamma_i     &= \alpha_i+\alpha_{i+1}=\eps_i-\eps_{i+2} \qquad \text{for }i=1,\dots,n-2, \\
\gamma_{n-1} &= \alpha_{n-1}+\alpha_n=\eps_{n-1},                                        \\
\gamma_n     &= \alpha_n=\eps_n.
\end{align*}
$\Sigma_X$ is a basis of a root system of type $\BBB_{[(n+1)/2]}\times\BBB_{[n/2]}$. The little Weyl group $W_X$ is generated by the reflections $r_{\gamma_1},\dots,r_{\gamma_n}$ associated to the spherical roots; observe that $r_{\gamma_i}=s_{\alpha_i}s_{\alpha_{i+1}}s_{\alpha_i}$ for $i<n$ and $r_{\gamma_n}=s_{\alpha_n}$. The very little Weyl group $\vlW_X$ is generated by $s_{\alpha_n}$.

We have $P=B$. One may choose a slice $Z\subset\OO_{\max}$ consisting of the points $x_a=(V_0,V_a)$, where $a=(a_1,\dots,a_n)\in(\CC^{\times})^n$ and
\begin{align*}
V_0&=\langle e_{-1},\dots,e_{-n}\rangle, \\
V_a&=\langle e_{-1}+a_1e_2,\dots,e_{-i}+a_ie_{i+1}-a_{i-1}e_{i-1},\dots, e_{-n}+a_ne_0-\tfrac{a_n^2}2e_n-a_{n-1}e_{n-1}\rangle,
\end{align*}
cf.\ \cite[9.1, case 10]{dbl.flag}. The coordinates $a_i$ on $Z$ have $T$-eigenweights $-\eps_i-\eps_{i+1}$ for $i<n$ and $-\eps_n$ for $i=n$.

It follows that the two $T^{\sigma}$-orbits in $Z^{\mu}$ or the two $B^{\sigma}$-orbits $\OO,\OO'$ in $\OO_{\max}^{\mu}$ are represented by $x_a$ with $a_i\in\RR^{\times}$ and distinguished by the sign of $a_1a_3a_5\cdots$. Let us assume that $\OO$ is the open $B^{\sigma}$-orbit containing the base point $x_0=x_{(1,\dots,1)}$.

To understand the action of the little Weyl group on $\Bopen(X^{\mu})$, we first note that the simple roots $\alpha_1,\dots,\alpha_{n-1}$ are of type (U) with respect to $\OO$ and $\OO'$, while $\alpha_n$ is of type (T2), because $\alpha_n=\gamma_n$ is the unique simple root proportional to a spherical root and both $\OO,\OO'$ are in one and the same $G^{\sigma}$-orbit (see Propositions \ref{prop:types-spherroots} and~\ref{prop:refl&G-orb}). In particular, $s_{\alpha_n}$ interchanges $\OO$ and $\OO'$.

Next we consider the action of $s_{\alpha_i}$ with $i<n$ on $\OO$. The group $(P_{\alpha_i})_{x_0}$ is contained in the common stabilizer of two maximal isotropic subspaces $V_0$ and $V_+=\langle e_1,\dots,e_n\rangle$, which is just the standard Levi subgroup of the maximal parabolic subgroup corresponding to $\alpha_n$. This Levi subgroup consists of the matrices $\diag(C,1,C^*)$, where $C\in GL_n\simeq GL(V_+)$ and $C^*$ is the matrix of the conjugate operator in $GL(V_0)$ under the inner product pairing between $V_0$ and $V_+$. We shall identify this Levi subgroup in $SO_{n+1,n}$ with $GL_n$. Under this identification, $(P_{\alpha_i})_{x_0}$ is the one-dimensional unipotent subgroup consisting of the matrices
$$
C=
\begin{pmatrix}
1 &              &  & \smash{\text{\raisebox{-0.4ex}{$\vdots$}}} &&&  \\
  & \smash\ddots &  & \smash\vdots &  &  \smash{\text{\huge$0$}}  &   \\
  &              & 1 &     -c      &  &                           &   \\
  &              &   &      1      &           \hdotsfor3             \\
  &              &   &      c      & 1 &                          &   \\
  & \smash{\text{\huge$0$}} &  &   &   &       \smash\ddots       &   \\
  &                         &  &   &   &                          & 1 \\
\end{pmatrix}
\text{\footnotesize $i$-th row}
$$
Hence $R_{x_0}=U^-(PGL_2(\RR))$ and $\9\OO{s_{\alpha_i}}=:\OO_i$ is the $B^{\sigma}$-orbit of the point $x_i=n_{\alpha_i}x_0$, where
$$
n_{\alpha_i}=\diag\biggl(\underbrace{1,\dots,1}_{i-1},
\begin{pmatrix}
0 & -1 \\
1 &  0 \\
\end{pmatrix},\underbrace{1,\dots,1}_{2n-2i-1},
\begin{pmatrix}
 0 & 1 \\
-1 & 0 \\
\end{pmatrix}
\underbrace{1,\dots,1}_{i-1}\biggr)
\in P_{\alpha_i}^{\sigma}.
$$
A computation shows that $(P_{\alpha_{i+1}})_{x_i}$ is the two-dimensional subgroup consisting of the matrices $\diag(C,1,C^*)$ with
$$
C=\begin{pmatrix}
1 &              &  & & \smash{\text{\raisebox{-0.4ex}{$\vdots$}}} &&&  \\
  & \smash\ddots &  & & \smash\vdots &  &  \smash{\text{\huge$0$}}  &   \\
  &              & 1 & t^{-1}-1 &  c  & &                           &   \\
  &              &   &  t^{-1}  &  c  &           \hdotsfor3            \\[2ex]
  &              &   &          &  t  &                            &&   \\
  &              &   &          & 1-t & 1                          &&   \\
  & \smash{\text{\huge$0$}} &&  &     &   &      \smash\ddots       &   \\
  &                         &&  &     &   &                         & 1 \\
\end{pmatrix}
\text{\raisebox{2.4ex}{\footnotesize $i$-th row}}
$$
whenever $i\le n-2$ and $(P_{\alpha_n})_{x_{n-1}}$ is the two-dimensional subgroup of $(2n+1)\times(2n+1)$ matrices
$$
\begin{pmatrix}
1 &              &   &          &               &     &                         &&                         \\
  & \smash\ddots &   &          &               &     &                         &&                         \\
  &              & 1 & t^{-1}-1 & \frac{t-1}2+c &     & \smash{\text{\huge$0$}} &&                         \\
  &              &   &  t^{-1}  &       c       &     &                         &&                         \\[2ex]
  &              &   &          &       t       &     &                         &&                         \\[2ex]
  &              &   &          &      1-t      &  1  &                     \hdotsfor3                     \\
  &       &&& -\frac{(1-t)^2}{2t} & 1-t^{-1} & t^{-1} & & \smash{\text{\llap{\raisebox{-1ex}{\Large$*$}}}} \\[2ex]
  &           & \smash{\text{\huge$0$}} &&& \smash{\text{\raisebox{0ex}{$\vdots$}}} && \smash\ddots &    \\[0.3ex] &&&&&  \smash{\text{\raisebox{0.7ex}{$\vdots$}}} &&& \smash{\!\!\!\!\!\text{\raisebox{1.2ex}{$\ddots$}}} \\
\end{pmatrix}
\text{\raisebox{-3.3ex}{\footnotesize $n+1$-st row}}
$$
(where the entries below the $n+1$-st row and to the right of the $n+1$-st column are uniquely determined by the assumption that the matrix is in $SO_{n+1,n}$). It follows that $\alpha_{i+1}$ is of type (T1) with respect to $\OO_i$ and $\OO_i$ is open in $P_{\alpha_{i+1}}^{\sigma}\OO_i$.

The same argument applies to $\OO'=B^{\sigma}x_{(-1,1\dots,1)}$ and finally yields the following picture for the action of reflection operators, which describes the action of $W_X$ and $\vlW_X$ on $\Bopen(X^{\mu})$:
$$
\xymatrix{
\OO \ar@{<->}[rr]^{s_{\alpha_n}=r_{\gamma_n}} \ar@{<->}[d]_{s_{\alpha_i}}
\POS (-3,0) \drop{\phantom0} \ar@{>}@(ul,dl)_{r_{\gamma_i}} &&
\OO' \ar@{<->}[d]^{s_{\alpha_i}} \POS (29,0) \drop{\phantom0} \ar@{>}@(ur,dr)^{r_{\gamma_i}} \\
\OO_i  \POS (-3,-15) \drop{\phantom0} \ar@{>}@(ul,dl)_{s_{\alpha_{i+1}}} &&
\OO_i' \POS (29,-15) \drop{\phantom0} \ar@{>}@(ur,dr)^{s_{\alpha_{i+1}}}
}
$$

\subsection{}
\label{ex:unord.pairs}

A modification of the previous example is given by replacing $H$ with its normalizer in $G$. A model for $X$ is now the space of \emph{unordered} pairs $\{V,V'\}$ of transversal maximal isotropic subspaces in $\CC^{2n+1}$. A $G$-equivariant real structure on $X$ is induced from the one in the previous example, but the real locus is bigger: $\{V,V'\}\in X^{\mu}$ if and only if either $V=\overline{V},V'=\overline{V'}$, or $V'=\overline{V}$. Accordingly, the group $G^{\sigma}=\Spin_{n+1,n}(\RR)$, which may be replaced with $SO_{n+1,n}(\RR)^{\circ}$, acts on $X^{\mu}$ with two orbits.

The weight lattice $\Xi(X)$ is spanned by $\Sigma_X=\{\gamma_1,\dots,\gamma_{n-1},2\gamma_n\}$ and has index 4 in the character lattice of $T$. The factors $m_1,m_2,\dots$, in the notation of Proposition~\ref{prop:paramet-borel-orb}, are $2,2,1,\dots,1$ if $n=4k$ or $4k+3$, and $4,1,1,\dots,1$ if $n=4k+1$ or $4k+2$. The groups $W_X$ and $\vlW_X$ are the same as above. We still have $P=B$.

In the notation of Example~\ref{ex:ord.pairs}, it is easy to see that $x_a=(V_0,V_a)$ is mapped to $(V_{-a/2},V_{a/2})$ by a suitable element $g\in U$. The set of all pairs $(V_a,V_{-a})$ is yet another slice for Example~\ref{ex:ord.pairs}, which has the advantage of being stable under the involution $(V,V')\mapsto(V',V)$ and therefore descends to a slice in our $X$, which we again denote by $Z$, by abuse of notation.

Denoting $z_a=\{V_a,V_{-a}\}\in Z$, we see that $z_a=z_b$ if and only if $a=\pm b$. Also note that $\overline{V_a}=V_{\bar{a}}$, hence $z_a\in Z^{\mu}$ if and only if $a\in\RR\cup\ii\RR$. The action of $T^{\sigma}$ preserves the two pieces of $Z^{\mu}$ given by $a\in\RR$ and $a\in\ii\RR$, respectively. Each piece is a single $T^{\sigma}$-orbit if $n=4k+1$ or $4k+2$ and splits in two $T^{\sigma}$-orbits differing by the sign of $a_1a_3a_5\cdots$ (which is well defined on $Z^{\mu}$) if $n=4k$ or $4k+3$. The action of $W_X$ and $\vlW_X$ on $\Bopen(X^{\mu})$ in the latter case is described by the following picture:
$$
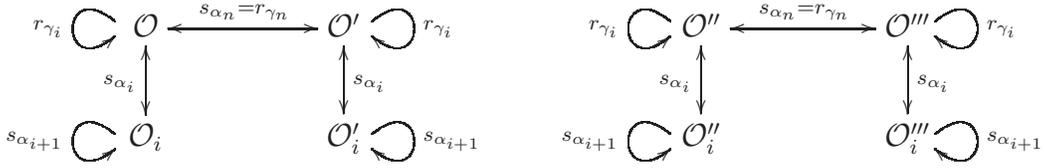
\xymatrix{
\OO \ar@{<->}[rr]^{s_{\alpha_n}=r_{\gamma_n}} \ar@{<->}[d]_{s_{\alpha_i}}
\POS (-3,0) \drop{\phantom0} \ar@{>}@(ul,dl)_{r_{\gamma_i}} &&
\OO' \ar@{<->}[d]^{s_{\alpha_i}} \POS (29,0) \drop{\phantom0} \ar@{>}@(ur,dr)^{r_{\gamma_i}} \\
\OO_i  \POS (-3,-15) \drop{\phantom0} \ar@{>}@(ul,dl)_{s_{\alpha_{i+1}}} &&
\OO_i' \POS (29,-15) \drop{\phantom0} \ar@{>}@(ur,dr)^{s_{\alpha_{i+1}}}
}
\qquad
\xymatrix{
\OO'' \ar@{<->}[rr]^{s_{\alpha_n}=r_{\gamma_n}} \ar@{<->}[d]_{s_{\alpha_i}}
\POS (-3,0) \drop{\phantom0} \ar@{>}@(ul,dl)_{r_{\gamma_i}} &&
\OO''' \ar@{<->}[d]^{s_{\alpha_i}} \POS (30,0) \drop{\phantom0} \ar@{>}@(ur,dr)^{r_{\gamma_i}} \\
\OO_i''  \POS (-3,-15) \drop{\phantom0} \ar@{>}@(ul,dl)_{s_{\alpha_{i+1}}} &&
\OO_i''' \POS (30,-15) \drop{\phantom0} \ar@{>}@(ur,dr)^{s_{\alpha_{i+1}}}
}
$$
Here $\OO,\OO'$ are the open $B^{\sigma}$-orbits corresponding to $a\in\RR$ and $\OO'',\OO'''$ correspond to $a\in\ii\RR$. If $n=4k+1$ or $4k+2$, then the same picture works under identifications $\OO=\OO'$, $\OO_i=\OO_i'$, $\OO''=\OO'''$, $\OO_i''=\OO_i'''$. Yet another difference with Example~\ref{ex:ord.pairs} is that $\alpha_n$ has type (N1) with respect to $\OO_{n-1},\OO_{n-1}',\OO_{n-1}'',\OO_{n-1}'''$ and type (N2) with respect to $\OO,\OO',\OO'',\OO'''$.

Examples \ref{ex:ord.pairs} and \ref{ex:unord.pairs} demonstrate that the groups $W_X$ and $\vlW_X$ may differ very much in size but have the same orbits on $\Bopen(X^{\mu})$.

\subsection{}
\label{ex:G/TU'}

Suppose that $G=G_0\times T_0$, where $G_0$ is a semisimple group and $T_0$ is a split maximal torus in $G_0$. Take a Borel subgroup $B=B_0\times T_0$, where $B_0$ is a Borel subgroup of $G_0$ containing $T_0$, and a maximal torus $T=T_0\times T_0$ in $B$. We embed $T_0$ ``diagonally'' into $T$ by taking $t\mapsto(t^2,t)$ and denote by $T_0'\subset T$ the image of this embedding.

Consider the homogeneous space $X=G/H$, where $H=T_0'\Ru{H}$ and $\Ru{H}=[U^-,U^-]$ is the derived subgroup of $U^-$. Here $\Sigma_X=\{\alpha_1,\dots,\alpha_l\}$ is the set of all simple roots of $(G,T)$ and $\vlW_X=W_X=W$. We get $\Bmax(X)=\{\Omax\}$ and $\Bmax(X^{\mu})=\Bopen(X^{\mu})$.

The open $(B\times H)$-orbit in $G$ is the subset
$$U\cdot U_{-\alpha_1}^{\times}\cdots U_{-\alpha_l}^{\times}\cdot T_0\cdot T_0'\cdot \Ru{H},$$
where $U_{-\alpha_i}^{\times}=U_{-\alpha_i}\setminus\{1\}$, and each element of this orbit has a unique presentation in the form $g=uu_{-\alpha_1}(a_1)\cdots u_{-\alpha_l}(a_l)t_0t'v$, where $u\in U$, $a_i\ne0$, $t_0\in T_0$, $t'\in T_0'$, $v\in\Ru{H}$. (Here and below in this subsection, $T_0$ denotes the first factor of $T$ sitting in $G_0$.) It follows that
$$\Omax=U\cdot U_{-\alpha_1}^{\times}\cdots U_{-\alpha_l}^{\times}\cdot T_0\cdot x_0$$
consists of the points $x=uu_{-\alpha_1}(a_1)\cdots u_{-\alpha_l}(a_l)t_0x_0$. 
For the slice, we may take
$$Z=U_{-\alpha_1}^{\times}\cdots U_{-\alpha_l}^{\times}\cdot T_0\cdot x_0.$$
The $T$-action on $Z$ is described as follows: given $x\in Z$ with $x=u_{-\alpha_1}(a_1)\cdots u_{-\alpha_l}(a_l)t_0x_0$ ($a_i\in\CC^\times$ and $t_0\in T_0$) and $t=(t_1,t_2)\in T$, we have
$$tx=u_{-\alpha_1}\bigl(a_1/\alpha_1(t_1)\bigr)\cdots u_{-\alpha_l}\bigl(a_l/\alpha_l(t_1)\bigr)\,t_1t_0t_2^{-2}x_0.$$

Now we consider the $B^{\sigma}$-orbits in $\Omax^{\mu}$. By Corollary~\ref{cor:struct-loc}, they are in bijection with the $T^{\sigma}$-orbits in $Z^{\mu}$. It is visible from the displayed formula right above that the $T^{\sigma}$-action on $Z^{\mu}$ preserves the signs of $a_i\cdot\alpha_i(t_0)$ ($i=1,\dots,l$). Assume for simplicity that $G_0$ is of adjoint type. Then these signs form a complete system of invariants separating open $B^{\sigma}$-orbits.

To understand the reflection operator on $\Bmax(X^{\mu})$ corresponding to $s_{\alpha_i}$, it suffices to look at open $B^{\sigma}$-orbits meeting a curve $u_{-\alpha_i}(s)\,x$ ($s\in\RR$) with $x\in Z^{\mu}$. A straightforward computation shows that
$$
u_{-\alpha_i}(s)\,x=u_{-\alpha_1}(a_1)\cdots u_{-\alpha_i}(a_i+s)\cdots u_{-\alpha_l}(a_l)\,t_0x_0,
$$
i.e., the action of $u_{-\alpha_i}(s)$ may alter the sign of $a_i\cdot\alpha_i(t_0)$ and preserves the signs of all other $a_j\cdot\alpha_j(t_0)$ ($j\ne i$).

It follows that the action of $s_{\alpha_i}$ on $\Bmax(X^{\mu})$ changes the $i$-th sign in the $l$-tuple of signs mentioned above. Thus $s_{\alpha_1},\dots,s_{\alpha_l}$ act on the set of $l$-tuples of signs as generators of an elementary Abelian 2-group of order $2^l$. This is not compliant with braid relations of type~$\AAA_2$. We deduce that neither the assertion of Theorem~\ref{thm:real-Knop-action} nor of Theorem~\ref{thm:main-little-Weyl} is true for $X$ unless the Dynkin diagram of $G$ contains no $\AAA_2$ pieces. Note that in this case all simple roots are of type (T2) with respect to any $B^{\sigma}$-orbit in $\Bmax(X^{\mu})$.


\end{document}